\documentclass[9pt]{amsart}
\usepackage{latexsym}
\usepackage{amstext}
\usepackage{amsmath}
\usepackage{amssymb}
\usepackage{amsthm}
\usepackage{url}
\usepackage{amscd}

\newtheorem{theorem}{Theorem}[section]
\newtheorem{lemma}[theorem]{Lemma}

\newtheorem{corollary}[theorem]{Corollary}

\theoremstyle{definition}
\newtheorem{definition}[theorem]{Definition}
\newtheorem{remark}[theorem]{Remark}

\def\CC{\mathbb C}
\def\PP{\mathbb P}

\def\ord{{\rm ord}}
\def \O {\mathcal{O}}

\DeclareMathOperator{\codim}{codim}

\begin{document}
\title[Greatest common divisors and Nevanlinna theory on algebraic tori]{Greatest common divisors of analytic functions and Nevanlinna theory on algebraic tori}

\author[Aaron Levin]{Aaron Levin }
\thanks{The first author was supported in part by NSF grant DMS-1352407.}
\thanks{The second author was supported in part by Taiwan's MoST grant 106-2115-M-001-001-MY2.}
\address{Department of Mathematics\\Michigan State University\\East Lansing, MI 48824, USA}
\email{adlevin@math.msu.edu}
\author{Julie Tzu-Yueh Wang}
\address{Institute of Mathematics\\ Academia Sinica\\Taipei 10617, Taiwan}
\email{jwang@math.sinica.edu.tw}
\date{}
\begin{abstract}
We study upper bounds for the counting function of common zeros of two meromorphic functions in various contexts.  The proofs and results are inspired by recent work involving greatest common divisors in Diophantine approximation, to which we introduce additional techniques to take advantage of the stronger inequalities available in Nevanlinna theory.  In particular, we prove a general version of a conjectural ``asymptotic gcd" inequality of Pasten and the second author, and consider moving targets versions of our results.

\end{abstract}
\thanks{2010\ {\it Mathematics Subject Classification}: Primary 32H30; Secondary 11J97}
\normalsize
\baselineskip=15pt

\maketitle

\section{Introduction}

We prove upper bounds for the counting function of common zeros of two meromorphic functions in various contexts.  A starting point for such results (in a geometric formulation) comes from the study of holomorphic curves in semi-abelian varieties by Noguchi, Winkelmann, and Yamanoi, who proved the following:

\begin{theorem}[Noguchi, Winkelmann, Yamanoi {\cite[Th.~5.1]{NWY} (see also \cite[\S 6.5]{NWbook})}]
\label{tNWY}
Let $f:\mathbb{C}\to A$ be a holomorphic map to a semi-abelian variety $A$ with Zariski-dense image. Let $Y$ be a closed subscheme of $A$ with $\codim Y\geq 2$ and let $\epsilon>0$.  
\begin{enumerate}
\item  Then
\begin{align*}
N_{f}(Y,r)\leq_{\rm exc} \epsilon T_f(r).
\end{align*}
\item There exists a compactification $\overline{A}$ of $A$, independent of $\epsilon$, such that for the Zariski closure $\overline{Y}$ of $Y$ in $\overline{A}$,
\begin{align*}
T_{\overline{Y},f}(r)\leq_{\rm exc} \epsilon T_f(r)
\end{align*}
\end{enumerate}
\end{theorem}

Here $N_f(Y,r)$ is a counting function associated to $f$ and $Y$, $T_{\overline{Y},f}(r)$ is a Nevanlinna characteristic (or height) function associated to $f$ and $\overline{Y}$, and $T_f(r)$ is any characteristic function associated to an appropriate ample line bundle (see \cite{NWY} for more discussion and Section \ref{Preliminary} for the relevant definitions from Nevanlinna theory).  The notation $ \le_{\rm exc}$ means that the estimate holds for all $r$ outside a set of finite Lebesgue measure, possibly depending on $\epsilon$. 

More generally, Noguchi, Winkelmann, and Yamanoi proved a result for $k$-jet lifts of holomorphic maps to semi-abelian varieties.  The case when $A$ is an abelian variety was proved by Yamanoi \cite{Yam}.  

A first goal of our work is to obtain a new proof of Theorem \ref{tNWY} when $A=(\mathbb{C}^*)^n$ is the complex algebraic torus.  In this case we obtain the following refinement of Theorem \ref{tNWY}.

\begin{theorem}\label{RefinementII}
Let $Y$ be a closed subscheme of $(\CC^*)^n$ of codimension at least $2$.  Let   ${\bf g}=(g_1,\cdots,g_n)$ be a holomorphic map from $\CC$ to $(\CC^*)^n$ with Zariski dense image (equivalently, $g_1,\dots,g_n $ are entire functions without zeros, and  $g_1^{i_1} \cdots g_n^{i_n} \notin\CC$ for any 
index set $(i_1,\dots,i_n)\in\mathbb Z^n\setminus \{(0,\dots,0)\}$).  Let $\epsilon>0$.
\begin{enumerate}
\item Then
 $$
 N_{\bf g}(Y,r)\le_{\operatorname{exc}}  \epsilon T_{\bf g}(r).
$$
\label{refa}
\item 
Let $X$ be a nonsingular projective toric compactification of $(\CC^*)^n$.  Let $\overline{Y}$ be the Zariski closure of $Y$ in $X$, and suppose that $\overline{Y}$ is in general position with the boundary of $(\CC^*)^n$ in $X$.  Then
\begin{align*}
T_{\overline{Y},\bf g}(r)\leq_{\operatorname{exc}} \epsilon T_{\bf g}(r).
\end{align*}
\label{refb}
\end{enumerate}
\end{theorem}

Here, we say that $\overline{Y}\subset X$ is in general position with the boundary $X\setminus (\CC^*)^n$ if $\overline{Y}$ does not contain any point of intersection of $n$ distinct irreducible components of $X\setminus (\CC^*)^n$.

Alternatively, the counting function $N_{\bf g}(Y,r)$ can be expressed as a counting function of common zeros of functions obtained by composing ${\bf g}$  with polynomials generating the defining ideal of $Y$.  In this context, we prove the following:

\begin{theorem}\label{gcdunit}
Let $F,\, G\in \mathbb C[x_1, \dots,x_n ]$ be nonconstant coprime  polynomials.
Let $g_1,\hdots,g_n$ be entire functions without zeros.  Assume that  $g_1^{i_1} \cdots g_n^{i_n} \notin\CC$ for any 
index set $(i_1,\dots,i_n)\in\mathbb Z^n\setminus \{(0,\dots,0)\}$. Let $\epsilon>0$.
\begin{enumerate}
\item  Then
$$
N_{\gcd}(F(g_1,\hdots,g_n), G(g_1,\hdots,g_n),r)\le_{\rm exc}  \epsilon\max_{1\le i\le n}\{ T_{g_i}(r)\}.
$$
\label{gcda}
\item If $F$ and $G$ do not both vanish at the origin $(0,\ldots, 0)$, then
$$
T_{\gcd}(F(g_1,\hdots,g_n), G(g_1,\hdots,g_n),r)\le_{\rm exc}  \epsilon\max_{1\le i\le n}\{ T_{g_i}(r)\}.
$$
\label{gcdb}
\end{enumerate}
\end{theorem}

Here, as will be discussed in more detail later, $N_{\gcd}(f,g,r)$ and $T_{\gcd}(f,g,r)$ are analogues of the greatest common divisor of two integers.  The function $N_{\gcd}(f,g,r)$ is simply the counting function of common zeros of $f$ and $g$.  

A particularly simple consequence of Theorem \ref{gcdunit} is the following result.

\begin{corollary}\label{fgunit}
Let $f$ and $g$ be multiplicatively independent  non-constant entire functions without zeros. 
Then for any $\epsilon >0$,
\begin{align}\label{corunit}
N_{\gcd}(f-1, g-1,r)\le_{\rm exc}  \epsilon\max \{ T_{f}(r), T_{g}(r)\}.
\end{align}
\end{corollary}

An elementary proof of Corollary \ref{fgunit} was previously obtained in \cite {PWgcd} by adapting the number-theoretical arguments of \cite{CZ2} to Nevanlinna theory.

It is natural to try to extend Theorem \ref{gcdunit} (and Corollary \ref{fgunit}) in an appropriate way to entire functions, or more generally, to meromorphic functions.  In this direction, by adapting the ideas  and  methods of Silverman \cite{Silverman2005generalized} to a pair of meromorphic functions, the following estimate was established by Pasten and the second author in \cite[Proposition 7.2] {PWgcd} under the assumption of Vojta's conjecture for a blow-up of $\mathbb{P}^1\times \mathbb{P}^1$  at a single point: If $f$ and $g$ are algebraically independent complex meromorphic functions, then for all $\epsilon >0$,  
\begin{align}\label{gcdcomplex}
N_{\gcd}&(f-1,g-1,r )\le_{\rm exc}  \,\epsilon\max\{  T_{f }( r), T_{g}( r)\}\cr
&+ \frac1{1+\epsilon/4}  (N_{f}( 0,r) +N_{g}( 0,r)+N_{f}(\infty,r) +N_{g}(\infty,r) ).
\end{align}

Again under the assumption of Vojta's conjecture (as for (\ref{gcdcomplex})), the following asymptotic gcd estimate is formulated in \cite[Proposition 7.4] {PWgcd}: If $f$ and $g$ are multiplicatively independent meromorphic functions, then for any $\epsilon >0$,  there exists $n_0$ such that for all $n\ge n_0$,
\begin{align}\label{asymgcdcomplex}
 N_{\gcd}(f^n-1, g^n-1, r)\le_{\rm exc}   \epsilon\max\{ T_{f^n}( r),  T_{g^n}( r)\}.
 \end{align}

A second goal of this article is to prove (unconditional) asymptotic gcd estimates in a more general context:

\begin{theorem}\label{gcdPn}
Let $F,\, G\in \mathbb C[x_1, \dots,x_n ]$ be nonconstant coprime  polynomials such that not both of them vanish at $(0,\dots,0)$.
Let $g_1,\hdots,g_n$ be  meromorphic functions such that $g_1^{i_1} \cdots g_n^{i_n} \notin\CC$ for any 
index set $(i_1,\dots,i_n)\in\mathbb Z^n\setminus \{(0,\dots,0)\}$.   Then for any $\epsilon>0$, there exists $k_0$ such that for all $k\ge k_0$,
\begin{enumerate}
\item   
$$
N_{\gcd}(F(g_1^k,\hdots,g_n^k), G(g_1^k,\hdots,g_n^k),r)\le_{\rm exc} \epsilon\max_{1\le i\le n}\{ T_{g_i^k}(r)\};
$$
\label{kgcda}
\item
$$
T_{\gcd}(F(g_1^k,\hdots,g_n^k), G(g_1^k,\hdots,g_n^k),r)\le_{\rm exc}  \epsilon\max_{1\le i\le n}\{ T_{g_i^k}(r)\},
$$
 if $g_1,\hdots,g_n$ are entire functions.
\label{kgcdb}
\end{enumerate}
\end{theorem}

In particular,  we prove the conjectured inequality \eqref{asymgcdcomplex}:

\begin{corollary}\label{fgagcd}
Let $f$ and $g$ be multiplicatively independent meromorphic functions. 
Then for any $\epsilon>0$, there exists $k_0$ such that for all $k\ge k_0$,
\begin{align*}
N_{\gcd}(f^k-1, g^k-1,r)\le_{\rm exc}  \epsilon\max \{ T_{f^k}(r), T_{g^k}(r)\}.
\end{align*}
\end{corollary}

When $f$ and $g$ are algebraically independent meromorphic functions, Corollary \ref{fgagcd} was recently obtained by Guo and the second author in \cite{GuoWang} with $\epsilon$ replaced by $\frac12+\epsilon$.  We refer to \cite {PWgcd} for further discussion of gcd problems for both complex and non-archimedean meromorphic functions. 

Even in the special case when $g_1,\ldots, g_n$ are complex polynomials, Theorem \ref{gcdPn} gives new results.

\begin{remark}
Let $F,\, G\in \mathbb C[x_1, \dots,x_n ]$ be nonconstant coprime  polynomials such that not both of them vanish at $(0,\dots,0)$.
Let $g_1,\ldots, g_n\in \CC[z]$ be complex polynomials such that $g_1^{i_1} \cdots g_n^{i_n} \notin\CC$ for any 
index set $(i_1,\dots,i_n)\in\mathbb Z^n\setminus \{(0,\dots,0)\}$.  It is elementary that in this case,
\begin{align}
\label{Ngcdpoly}
N_{\gcd}(F(g_1^k,\hdots,g_n^k), G(g_1^k,\hdots,g_n^k),r)=(\deg \gcd(F(g_1^k,\hdots,g_n^k), G(g_1^k,\hdots,g_n^k)))\log r+O(1)
\end{align}
and
\begin{align*}
\max_{1\le i\le n}\{ T_{g_i^k}(r)\}=k\max_{1\le i\le n}\{ \deg g_i\}\log r+O(1), 
\end{align*}
where the $\gcd$ on the right-hand side of \eqref{Ngcdpoly} is the greatest common divisor in the polynomial ring $\CC[z]$.  Then Theorem \ref{gcdPn} implies that for any $\epsilon>0$, there exists $k_0$ such that for all $k\ge k_0$,
\begin{align*}
\deg \gcd(F(g_1^k,\hdots,g_n^k), G(g_1^k,\hdots,g_n^k))<\epsilon k.
\end{align*}
More generally, Theorem \ref{gcdPn} gives a similar statement for rational functions $g_1,\ldots, g_n\in \CC(z)$.  When $n>2$, the only previous result in this direction appears to be a result of Ostafe \cite[Th.~1.3]{Ostafe}, which considers special polynomials such as $F=x_1\cdots x_r-1, G=x_{r+1}\cdots x_n-1$, but proves a stronger uniform bound independent of $k$. It is noted in \cite{Ostafe} that it appears to be difficult to extend the techniques used there to obtain results for general $F$ and $G$.  In the $n=2$ case, previous results include the original theorem of Ailon-Rudnick \cite{AR} in this setting and extensions of Ostafe \cite{Ostafe} (both with uniform bounds).

\end{remark}

We will use Theorem \ref{gcdunit} and Theorem \ref{gcdPn} to solve the following quotient problem, which can be considered an analogue of the``Hadamard quotient theorem" for recurrence sequences proved  by van der Poorten (see \cite{Rumely} and \cite{vdP}, and see \cite{CZ2002} and \cite{Survey} for an overview of the existing improvements).  
To state the result we make the following definitions: Let $G\in \mathbb C[x_1, \dots,x_n ]$ be a nonconstant polynomial  such that $G(0,\dots,0)\ne 0$. Since $G$ has a non-zero constant term, after arranging the index set in some order, we may write  
$$
G=a_{{\bf i}(0)}+\sum_{j=1}^{\ell}a_{{\bf i}(j)}{\bf x}^{{\bf i}(j)},
$$      
where $a_{{\bf i}(j)}\ne 0$ for $0\le j\le \ell$.  Then for entire functions $g_1,\hdots,g_n$, we let 
$$
\frak g_G:=(1,{\bf g}^{{\bf i}(1)},\hdots,{\bf g}^{{\bf i}(\ell)}):\CC\to \mathbb P^{\ell}.
$$ 
For  functions $h_i:\mathbb R_{\ge 0}\to \mathbb R_{\ge 0}$, $i=1,2$, we write $h_1(r)\asymp  h_2(r)$ if there exist positive numbers $a$ and $b$ such that $ah_1(r)\le h_2(r)\le bh_1(r)$ for all sufficiently large $r$.

\begin{corollary}\label{divisibility}
Let $F,\, G\in \mathbb C[x_1, \dots,x_n ]$ be nonconstant coprime  polynomials  such that $G(0,\dots,0)\ne 0$.
Let $g_1,\hdots,g_n$ be entire functions such that $T_{\frak g_G}(r)\asymp  \max_{1\le i\le n}\{ T_{g_i}(r)\}$.
\begin{enumerate}
\item If $F(g_1^k,\hdots,g_n^k)/G(g_1^k,\hdots,g_n^k)$ is an entire functions for infinitely many positive integers $k$; or\label{diva}
\item $g_1,\hdots,g_n$ are entire functions without zeros and $F(g_1,\hdots,g_n)/G(g_1,\hdots,g_n)$ is an entire function,\label{divb}
\end{enumerate}
then
there exists an index set $(i_1,\dots,i_n)\in\mathbb Z^n\setminus \{(0,\dots,0)\}$ such that $g_1^{i_1} \cdots g_n^{i_n} \in\CC$.
\end{corollary}
\begin{remark}
The case where $(g_1^k-1)/(g_2^k-1)$ is entire for infinitely many positive integers $k$ (i.e., $F=-1+ x_1$ and $G=-1+x_2$) is treated in \cite{GuoWang} under the assumption that $T_{g_1}(r)\asymp T_{g_2}(r)$.   The case $G=1+a_1x_1+\hdots+a_nx_n$, $F=b_0+b_1x_{n+1}+\hdots+b_mx_{n+m}$ ($a_ib_j\ne 0$ for $1\le i\le n$ and $0\le j\le m$) with the assumption that $\max\{T_{g_1}(r),\hdots,T_{g_n}(r)\}\asymp\max\{T_{g_{n+1}}(r),\hdots,T_{g_{n+m}}(r)\}$
is studied in \cite{Guo} when $a_i$ and $b_j$ are constants and  in \cite{Guo2} when $a_i$ and $b_j$ are ``small functions", with a proof obtained by adapting the argument of \cite{CZ2002}, where Corvaja and Zannier proved a stronger version of the Hadamard quotient theorem through a sophisticated application of Schmidt's subspace theorem.   The result in \cite{Guo} can be obtained from Corollary \ref{divisibility} since in this situation 
$$
T_{G(g_1,\hdots,g_n)}(r)=T_{1+a_1g_1+\dots+a_ng_n}(r)=T_{(1,g_1,\hdots,g_n)}(r)+O(1)
$$ 
and 
$$ 
\max\{T_{g_1}(r),\hdots,T_{g_n}(r)\}\le T_{(1,g_1,\hdots,g_n)}(r)\le n\max\{T_{g_1}(r),\hdots,T_{g_n}(r)\}.
$$
\end{remark}
\begin{remark}
The growth condition on characteristic functions is essential, at least in part (b).  For instance, let $g(z)=\exp(2\pi iz)$ and $f(z)=\exp(2\pi ip(z))$, where $p(x)\in\mathbb Z[x]$ is a polynomial of degree at least $2$. Then $(g-1)|(f-1)$, but $f$ and $g$ are algebraically independent.
\end{remark}

Finally, we will consider the case when the coefficients of $F$ and $G$ are functions.
More precisely,  let ${\bf g}$ be a holomorphic map from $\CC$ to  $\PP^n$.   We say a meromorphic function $a$ is a {\it small function} with respect to ${\bf g}$ if $T_{a}(r)={\rm o} (T_{\bf g}(r))$.
Let $K_{\bf g}$ be the field containing all small functions with respect to ${\bf g}$. 
Let $F,G\in K_{\bf g}[x_1,\hdots,x_n]$ be nonconstant coprime polynomials.
We establish results analogous to Theorem~\ref{gcdunit}, Theorem~\ref{gcdPn}, and Corollary~\ref{divisibility} in this situation.  The results are stated in Section \ref{moving}.

The proofs of Theorem \ref{RefinementII} and Theorem \ref{gcdunit} are inspired by recent work of the first author \cite{levin2017} on analogous inequalities in Diophantine approximation involving greatest common divisors of multivariable polynomials evaluated at $S$-unit arguments.  The proofs of Theorem \ref{gcdPn} and Corollary \ref{fgagcd} go well beyond what is possible in the arithmetic setting, as they take advantage of inequalities involving truncated counting functions and Wronskian terms, whose analogues are still largely conjectural in Diophantine approximation.  

It will be useful to give a brief overview of the analogous results in Diophantine approximation involving greatest common divisors.  The first general result in this direction is due to Bugeaud, Corvaja, and Zannier, who in 2003 proved:

\begin{theorem}[Bugeaud, Corvaja, Zannier \cite{BCZ}]
\label{tBCZ}
Let $a,b\in \mathbb{Z}$ be multiplicatively independent integers.  Then for every $\epsilon>0$, 
\begin{align*}
\log \gcd (a^n-1,b^n-1)\leq \epsilon n
\end{align*}
for all but finitely many positive integers $n$.
\end{theorem}

Let $k$ be a number field, $M_k$ the set of places of $k$, $S\subset M_k$ a finite set of places containing the archimedean places, $\O_{k,S}$ the ring of $S$-integers of $k$, and $\O_{k,S}^*$ the group of $S$-units.  Let $\mathbb{G}_m^n$ be the $n$-dimensional algebraic torus and $\mathbb{G}_m^n(\O_{k,S})=(\O_{k,S}^*)^n$.  To state more general results, for $\alpha,\beta\in k$, we define the generalized logarithmic greatest common divisor and the $\gcd$ counting function (depending on a choice of $S$), respectively, by
\begin{align}
\log \gcd(\alpha,\beta)&=-\sum_{v\in M_k}\log^- \max\{|\alpha|_v,|\beta|_v\}=h([1:\alpha:\beta])-h([\alpha:\beta]),\label{loggcd}\\
N_{\gcd, S}(\alpha,\beta)&=-\sum_{v\in M_k\setminus S}\log^- \max\{|\alpha|_v,|\beta|_v\},\label{loggcd2}
\end{align}
where $|\cdot|_v$ is an appropriately normalized absolute value associated to $v$, $h$ is the standard (absolute logarithmic) Weil height on projective space, and $\log^-x=\min\{0,\log x\}$ (see \cite{levin2017} for details).  After work of Corvaja and Zannier \cite{CZ2} in the $2$-dimensional case, the first author generalized Theorem \ref{tBCZ} as follows.

\begin{theorem}[Corvaja, Zannier \cite{CZ2} ($n=2$), Levin \cite{levin2017} ($n>2$)]
\label{tLCZ}
Let $F,G\in k[x_1,\ldots, x_n]$ be coprime polynomials.  Let $\epsilon>0$.

\begin{enumerate}
\item  There exists a finite union $Z$ of translates of proper algebraic subgroups of $\mathbb{G}_m^n$ such that
\begin{align*}
N_{\gcd, S}(F(u_1,\ldots, u_n), G(u_1,\ldots, u_n))<\epsilon \max\{h(u_1),h(u_2),\ldots, h(u_n)\}
\end{align*}
for all $(u_1,\ldots, u_n)\in \mathbb{G}_m^n(\O_{k,S})\setminus Z$.
\item  Suppose additionally that not both of $F$ and $G$ vanish at the origin $(0,0,\ldots, 0)$.  Then there exists a finite union $Z$ of translates of proper algebraic subgroups of $\mathbb{G}_m^n$ such that
\begin{align*}
\log \gcd(F(u_1,\ldots, u_n), G(u_1,\ldots, u_n))<\epsilon \max\{h(u_1),h(u_2),\ldots, h(u_n)\}
\end{align*}
for all $(u_1,\ldots, u_n)\in \mathbb{G}_m^n(\O_{k,S})\setminus Z$.
\end{enumerate}
\end{theorem}

One can also give a geometric version of Theorem \ref{tLCZ}.

\begin{theorem}
\label{tL}
Let $\mathbb{G}_m^n\subset X$ be a nonsingular projective toric variety of dimension $n$, and let $Y$ be a closed subscheme of $X$ of codimension at least $2$, both defined over a number field $k$. Let $A$ be a big divisor on $X$.  Let $S$ be a finite set of places of $k$ containing the archimedean places.  Let $\epsilon>0$.
\begin{enumerate}
\item There exists a finite union $Z$ of translates of proper algebraic subgroups of $\mathbb{G}_m^n$ such that
\begin{align*}
N_{Y,S}(P)\leq \epsilon h_A(P)+O(1)
\end{align*}
for all $P\in \mathbb{G}_m^n(\O_{k,S})\setminus Z\subset X(k)$.
\item Suppose that $Y$ is in general position with the boundary of $\mathbb{G}_m^n$ in $X$.  Then there exists a finite union $Z$ of translates of proper algebraic subgroups of $\mathbb{G}_m^n$ such that
\begin{align*}
h_Y(P)\leq \epsilon h_A(P)+O(1)
\end{align*}
for all $P\in \mathbb{G}_m^n(\O_{k,S})\setminus Z\subset X(k)$.
\end{enumerate}
\end{theorem}

All of the described arithmetic inequalities rely on Schmidt's Subspace Theorem in Diophantine approximation.  The proofs of Theorem \ref{RefinementII} and Theorem \ref{gcdunit} are adapted from the proofs in \cite{levin2017}, and use the analogue of Schmidt's Subspace Theorem in Nevanlinna theory: Cartan's second main theorem (in a form due to Vojta).  The proofs of the asymptotic gcd inequalities of Theorem~\ref{gcdPn} and Corollary~\ref{fgagcd} combine techniques from \cite{levin2017} along with new ideas in order to take advantage of the stronger error terms known in the second main theorem in Nevanlinna theory.

After giving the relevant background material in the next section, in Section \ref{sKey} we prove the key technical results underlying the paper.  In Section \ref{sMain}, we apply these results to prove our main theorems. In Section \ref{moving}, we review some fundamental results in Nevanlinna theory with moving targets, and we state and prove the main theorems when the coefficients of $F$ and $G$ are small functions.

\section{Background material}\label{Preliminary}

\subsection{Nevanlinna theory over $\CC$}\label{Nevanllina}

We will set up some notation and definitions in
 Nevanlinna theory for complex meromorphic functions and recall some basic results.
We refer to   \cite[Chapter VI ]{La} or \cite[Chapter 1]{ru2001nevanlinna} for details.

Let   $f$ be a meromorphic function  and   $z\in \mathbb C$. Define $v_z(f):={\rm ord}_z(f)$,
$$v_z^+ (f):=\max\{0,v_z(f)\}, \quad\text{and }\quad
  v_z^- (f):=-\min\{0,v_z(f)\}.
$$
Let $n_f(\infty,r)$ (respectively,  ${n}^{(Q)}_{f}(\infty,r)$) denote the number of poles of $f$ in $\{z:|z|\le r\}$, counting multiplicity (respectively, ignoring multiplicity larger than $Q\in\mathbb N$). The  {\it counting function} and {\it truncated counting function} of $f$ of order $Q$ at $\infty$  are  defined respectively by
\begin{align*}
N_f(\infty,r)&:=\int_0^r\frac{n_f(\infty,t)-n_f(\infty,0)}t dt+n_f(\infty,0)\log r\\
&=\sum_{0<|z|\le r } v_z^- (f)\log |\frac{r}{z}|+v_0^- (f)\log r,
\end{align*}
and
\begin{align*}
N^{(Q)}_f(\infty,r)&:=\int_0^r\frac{n^{(Q)}_f(\infty,t)-n^{(Q)}_f(\infty,0)}t dt+n^{(Q)}_f(\infty,0)\log r\\
&=\sum_{0<|z|\le r } \min\{Q,v_z^- (f)\}\log |\frac{r}{z}|+\min\{Q,v_0^- (f)\}\log r.
\end{align*} 
Then define the {\it counting function} $N_f(r,a)$ and the {\it truncated counting function} $N^{(Q)}_f(r,a)$ for $a\in\CC$ as
$$
N_f(a,r):=N_{1/(f-a)}(r, \infty)\quad\text{and}\quad N^{(Q)}_f(a,r):=N^{(Q)}_{1/(f-a)}(\infty,r).
$$
The  {\it proximity function} $m_f(\infty,r)$ is defined by
$$
m_f(\infty,r):=\int_0^{2\pi}\log^+|f(re^{i\theta})|\frac{d\theta}{2\pi},
$$
where $\log^+x=\max\{0,\log x\}$ for  $x\ge 0$. For any $a\in \CC,$ the {\it proximity function} $m_f(a,r)$ is defined by
$$m_f(a,r):=
m_{1/(f-a)}(\infty,r).
$$
Finally,  the {\it characteristic function} is defined by
$$
T_f(r):=m_f(\infty,r)+N_f(\infty,r).
$$
We recall the following version of Jensen's formula.
\begin{lemma}\label{Jensen}
Let $f$ be a meromorphic function on $\{z: |z|\le r\}$ which is not the zero function.  Then
\begin{align*}
\int_0^{2\pi}\log|f(re^{i\theta})|\frac{d\theta}{2\pi}
&=N_f(r,0)- N_f(r,\infty)+\log |c_f|,
\end{align*}
where $c_f$ is the leading coefficient of $f$ expanded as Laurent series in $z$, i.e.,
$f=c_fz^m+\cdots $ with $c_f\ne 0$.
\end{lemma}

Jensen's formula implies the first main theorem of Nevanlinna theory.
\begin{theorem}[First Main Theorem]\label{Cfirstmain}  Let $f$ be a non-constant  meromorphic function on $\CC$.  Then for every $a\in\CC$, and any positive real number $r$,
$$
m_f(a,r)+N_f(a,r)=T_f(r)+O(1),
$$
where $O(1)$ is independent of $r$.
\end{theorem}


\subsection{ Nevanlinna theory for Cartier divisors}

We recall some notation and properties  from \cite[Section 9 and 12]{Vojta2}. 
Let $D$ be a Cartier divisor on a complex variety $X$.   A Weil function for $D$ is a function   $\lambda_D:(X\setminus \operatorname{Supp}D)(\mathbb{C})\to \mathbb{R}$ such that for all $x\in X(\mathbb{C})$, there is an open neighborhood $U$ of $x$ in $X$, a nonzero function $f\in K(X)$ such that $D|_U=(f)$, and a continuous function $\alpha:U(\mathbb{C})\to \mathbb{R}$ such that 
$$
\lambda_D(x)=-\operatorname{log}|f(x)|+\alpha(x)
$$ 
for all $x\in (U\setminus\operatorname{Supp})(\mathbb{C})$.
 We note that when $D$ is effective $\lambda_D$ can be extended to a function $X \to \mathbb{R}\cup\{\infty\}$.
 Let $f:\mathbb{C}\to X$ be a holomorphic map whose image is not contained in the support of the divisor $D$ on $X$.  The proximity function of $f$ with respect to $D$ is defined by 
 $$
 m_f(D,r)=\int_0^{2\pi}\lambda_D(f (re^{i\theta}))\frac{d\theta}{2\pi}.
 $$
 Let $n_f(D,t)$ (respectively, $n^{(Q)}_f(D,t)$) be the number of zeros of $\rho\circ f$ inside $\{|z|<t\}$, counting multiplicity, (respectively, ignoring multiplicity larger than $Q\in\mathbb N$) with $\rho$ a local defining function for $D$.
The  {\it counting function} and {\it truncated counting function} of $f$ of order $Q$ at $\infty$  are  defined, respectively, by
$$
N_f(D,r) = \int_1^r \dfrac{n_f(D,t) }{t} dt,
$$
and 
$$
N^{(Q)}_f(D,r) = \int_1^r \dfrac{n^Q_f(D,t) }{t} dt.
$$
The characteristic function relative to $D$ is defined, up to $O(1)$, by 
$$
T_{D,f}(r)=m_f(D,r)+N_f(D,r).
$$
The following is the first main theorem.
\begin{theorem}
 Let $D$ and $D'$ be Cartier divisors on $X$ whose supports do not contain the image of $f$, and suppose that $D'$ is linearly equivalent to $D$. Then $$T_{D',f}(r)=T_{D,f}(r)+O(1).$$
\end{theorem}    
In particular, let $D$ be a hypersurface in $\mathbb P^n(\CC)$ defined by a homogeneous polynomial $F$ of degree $d$.    The Weil function for $D$ can be taken as 
\begin{align*}
\lambda_D({\bf x})=-\log \frac{|F({\bf x})|}{\max\{|x_0|,\dots ,|x_n|\}^d} 
\end{align*}
for  ${\bf x}=[x_0:\cdots:x_n]\in\mathbb P^n(\CC)$.
Let  ${\mathbf f}: \CC \rightarrow \mathbb P^n(\CC)$  be a holomorphic map and $(f_0, \dots ,f_n)$ be a reduced representation of ${\mathbf f}$, i.e. $f_0,\dots, f_n$ are  entire functions on $\CC$ without common zeros such that for all $z\in \CC$ we have ${\mathbf f}(z) = [f_0(z):\cdots : f_n(z)]$. The   {\it characteristic function}  $T_{\mathbf f}(r)$ is defined by
$$
T_{\mathbf f}(r) =   \int_0^{2\pi} \log\|{\mathbf f}(re^{i\theta})\|\frac{d\theta}{2\pi},
$$
where $\|{\mathbf f}(z)\| = \max\{|f_0(z)|,\dots ,|f_n(z)|\}$.
This definition is independent, up to an additive constant, of the choice of the reduced representation of ${\mathbf f}$.
In this context, the First Main Theorem reads as follows. 
\begin{theorem}\label{FMT}
Let  ${\mathbf f}: \CC \rightarrow \mathbb P^n(\CC)$  be a holomorphic map, and let $D$ be a hypersurface in $\mathbb P^n(\CC)$ 
of degree $d$.  If $f(\CC)\not\subset D$, then for $r>0$,
 \begin{align*} 
 dT_{\mathbf f}(r) =m_{\mathbf f}(r,D)+N_{\mathbf f}(r,D)+O(1),
 \end{align*}
where $O(1)$ is bounded independently of $r$.
\end{theorem}

We will make use of the following elementary inequality.

\begin{lemma}
\label{gineq}
Let $g_1,\ldots, g_n$ be meromorphic functions.  Let ${\bf g}:=[1:g_1:\cdots:g_n]:\CC\to\PP^n.$  Then
\begin{align*}
T_{\bf g}(r)\leq \sum_{i=1}^nT_{g_i}(r)+O(1).
\end{align*}
\end{lemma}

\begin{proof}
Let $H$ be the hyperplane in $\PP^n$ defined by $x_0=0$.  Then clearly
\begin{align*}
\lambda_H({\bf g(z)})=\log\max\{1, |g_1(z)|,\ldots, |g_n(z)|\}\leq \sum_{i=1}^n \log^+|g_i(z)|
\end{align*}
and
\begin{align*}
n_{\bf g}(H,r)\leq \sum_{i=1}^n n_{g_i}(\infty,r).
\end{align*}
After integrating, it follows from the definitions that
\begin{align*}
T_{\bf g}(r)&=m_{\bf g}(H,r)+N_{\bf g}(H,r)+O(1)\leq \sum_{i=1}^n \left(m_{g_i}(\infty,r)+N_{g_i}(\infty,r)\right)+O(1)\\
&\leq \sum_{i=1}^n T_{g_i}(r)+O(1).
\end{align*}

\end{proof}

We now recall the following general form of the second main theorem from  \cite[Theorem A3.1.3]{ru2001nevanlinna} which was proved by Vojta in \cite[Theorem 1]{vojta1997} and Ru in  \cite[Theorem 2.3]{ru1997general}.
\begin{theorem}
\label{gsmt}
    Let $f=(f_0,\hdots,f_n):\mathbb{C}\to\mathbb{P}^n(\mathbb{C})$ be a holomorphic curve whose image is not contained in any proper linear subspace, and where $f_0,\dots,f_n$ are entire functions with no common zeros. Let $H_1,\dots,H_q$ be arbitrary hyperplanes in $\mathbb{P}^n(\mathbb{C})$. Denote by $W(f)$ the Wronskian of $f_0,\dots,f_n$. Then for any $\varepsilon>0$, we have the inequality 
    $$\int_0^{2\pi} \max_J \sum_{k\in J}\lambda_{H_k}(f(re^{i\theta}))\frac{d\theta}{2\pi}+N_{W(f)}(0,r)\leq_{\operatorname{exc}} (n+1+\varepsilon)T_f(r)+o(T_f(r)),$$  where the maximum is taken over all subsets $J$ of $\{1,\dots, q\}$ such that the hyperplanes $H_j$,  $j\in J$, are in general position.
\end{theorem}
Finally, we recall Cartan's second main theorem with truncated counting functions. (See \cite[Theorem A3.2.2]{ru2001nevanlinna}.)
\begin{theorem}
\label{tsmt} Let $H_1,\dots,H_q$ be  hyperplanes in $\mathbb{P}^n(\mathbb{C})$ in general position. 
    Let $f=(f_0,\hdots,f_n):\mathbb{C}\to\mathbb{P}^n(\mathbb{C})$ be a holomorphic curve whose image is not contained in any proper linear subspace. Then for any $\varepsilon>0$, we have the following inequality: 
    $$
     (q-n-1-\varepsilon)T_f(r)\le_{\rm exc} \sum_{i=1}^q N_f^{(n)}(H_i,r).$$   
\end{theorem}


\subsection{Counting functions and proximity functions for closed subschemes}
We first recall some basic properties of  Weil functions associated to closed subschemes from \cite[Section 2]{Sil}.
 Let $Y$ be a closed subscheme on a projective variety $X$ defined over $\CC$.
Then one can associate to it a function
$$
\lambda_{Y}: X(\mathbb{C})\setminus \operatorname{Supp}(Y)\to \mathbb R
$$
satisfying some functorial properties (up to a O(1)) analogous to the arithmetic case as described 
in \cite[Theorem 2.1]{Sil}.
Intuitively, for each point $P\in X(\mathbb{C})$,
$$
\lambda_{Y}( P)=-\log(\text{distance from $P$ to $Y$}).
$$
The following lemma will be used to define Weil functions for closed subschemes.
\begin{lemma}\label{representation}
Let $Y$ be a closed subscheme of $X$.  There exist effective divisors $D_1,\cdots,D_r$ such that 
$$
Y=\cap_{i=1}^r D_i.
$$
\end{lemma}
\begin{proof}
See Lemma 2.2 from \cite{Sil}.
\end{proof}
\begin{definition}Let $X$ be a projective variety over $\CC$ and let $Y\subset X$ be a closed subscheme of $X$.
We define the Weil function for $Y$  as 
\begin{align*}
\lambda_{Y}=\min_i \{\lambda_{D_i}\},
\end{align*}
where $Y=\cap D_i$ (such $D_i$ exist according to the above lemma).
\end{definition}

As usual, the Weil function $\lambda_Y$ is well-defined up to $O(1)$.
Let $f:\CC\to X$ be an analytic map.  The {\it proximity function} of $f$ with respect to $Y$ is defined by 
\begin{align*}
m_f(Y,r) = \int_0^{2\pi} \lambda_Y (f(re^{i\theta})) \frac{d \theta}{2\pi};
\end{align*}
the {\it counting function} of $f$ with respect to $Y$ is  defined by 
\begin{align*}
N_f(Y,r) = \int_0^r \dfrac{n_f(Y,t)-n_f(Y,0) }{t} dt+ n_f(Y,r)\log r,
\end{align*}
where $n_f(Y,t)$ is the minimum of the number of zeros of $\rho_i\circ f$, $1\le i\le m$ inside $\{|z|<t\}$, counting multiplicity, with $\rho_i$, $1\le i\le m$, being   local defining functions of $Y$.  We note that the definition of the counting function extends without difficulty to the case when $Y$ is a closed subscheme of a quasi-projective variety.
The {\it  characteristic function} of $f$ with respect to $Y$ is  defined by
\begin{align*}
T_{Y,f}(r) =m_f(Y,r)+N_f(Y,r) .
\end{align*}
The proximity functions and counting functions of closed subschemes of $X$ satisfy additivity and functoriality properties as in the classical setting of Cartier divisors. 

Recall the following definition which gives an analogue of the notion of gcd in the context of meromorphic functions.
Let $f$ and $g$ be meromorphic functions. We write
$$
\begin{aligned}
n(f,g,r)&:=\sum_{|z|\le r} \min\{v^+_z(f),v^+_z(g)\}
\end{aligned}
$$
and
$$
\begin{aligned}
N_{\gcd}(f,g,r)&:=\int_0^r\frac{n(f,g,t)-n(f,g,0)}{t}dt + n(f,g,0)\log r.
\end{aligned}
$$
Let $F, G\in\mathbb C[x_0, \cdots,x_n ]$ be coprime homogeneous polynomials and let 
$Y$ be the closed subscheme of $\PP^n$ defined by the ideal $I=(F, G)$.   
Let ${\mathbf f}=(f_0, \dots ,f_n): \CC \rightarrow \mathbb P^n(\CC)$, where $f_0,\hdots,f_n$ are entire functions with no common zeros.  The counting function of $Y$ agrees with the gcd counting function of $F({\mathbf f})$ and $G({\mathbf f})$, i.e.,
$$
 N_{\mathbf f}(Y, r)=N_{\gcd}(F({\mathbf f}),G({\mathbf f}),r)+O(1).
$$

We also define gcd proximity and characteristic functions, in analogy with \eqref{loggcd} and \eqref{loggcd2}.  Let
\begin{align*}
m_{\gcd}(f,g,r)&:=-\int_0^{2\pi}\log^-\max\{|f(re^{i\theta})|,|g(re^{i\theta})|\} \frac{d\theta}{2\pi},\\
T_{\gcd}(f,g,r)&:=T_{[1:f:g]}(r)-T_{[f:g]}(r).
\end{align*}

We have the expected relationship between $m_{\gcd},N_{\gcd},$ and $T_{\gcd}$.

\begin{lemma}
\label{gcdfirstmain}
Let $f$ and $g$ be meromorphic functions.  Then
\begin{align*}
T_{\gcd}(f,g,r)=m_{\gcd}(f,g,r)+N_{\gcd}(f,g,r)+O(1).
\end{align*}
\end{lemma}

\begin{proof}
Let $h_0$ be an entire function such that $(h_0,fh_0,gh_0)$ is a reduced representation of $[1:f:g]$, i.e., $h_0,fh_0,$ and $gh_0$ are entire and have no common zeros.  Similarly, let $h_1$ be an entire function such that $[fh_0/h_1,gh_0/h_1]$ is a reduced representation of $[f:g]$.  Then from the definitions,
\begin{multline*}
T_{\gcd}(f,g,r)= \int_0^{2\pi} \log\max\{|h_0(re^{i\theta})|,|(fh_0)(re^{i\theta})|,|(gh_0)(re^{i\theta})|\}|\frac{d\theta}{2\pi}\\
-\int_0^{2\pi} \log\max\{|(fh_0/h_1)(re^{i\theta})|,|(gh_0/h_1)(re^{i\theta})|\}\frac{d\theta}{2\pi}.
\end{multline*}
It is elementary that for any real numbers $a,b>0$,
\begin{align*}
\log\frac{\max\{1,a,b\}}{\max\{a,b\}}=-\log^-\max\{a,b\}.
\end{align*}
Then we find that
\begin{align*}
T_{\gcd}(f,g,r)&= -\int_0^{2\pi} \log^-\max\{|f(re^{i\theta})|,|g(re^{i\theta})|\}\frac{d\theta}{2\pi}+\int_0^{2\pi} |h_1(re^{i\theta})|\frac{d\theta}{2\pi}\\
&=m_{\gcd}(f,g,r)+N_{h_1}(0,r)+O(1),
\end{align*}
where the second line follows from the definition of $m_{\gcd}$ and Jensen's formula.  To complete the proof, we note that from its definition, one easily finds that $N_{h_1}(0,r)=N_{\gcd}(f,g,r)$.
\end{proof}

\subsection{Polynomial rings and monomial orderings}\label{algebra}
Let $A=\CC[x_0,\hdots,x_n]$ be the polynomial ring in $n+1$ variables over $\CC$.  
For ${\bf i}=(i_0,\hdots,i_n)\in\mathbb N^{n+1}$, we define
$$
{\bf x}^{\bf i}= x_0^{i_1}\cdots x_n^{i_n},
$$
and write
$$ 
|{\bf i}|:=i_0+\cdots+  i_n=\deg  {\bf x}^{\bf i}.
$$
Let 
$$
{\rm Mon}={\rm Mon}(A)=\{{\bf x}^{\bf i}: {\bf i}\in{\mathbb N^{n+1}}\}
$$
be the set of monomials of $A$.  We note that we use the convention that ${\mathbb N}$ is the set of nonnegative integers.

Recall that a monomial ordering on $A$ is a total ordering $>$ on {\rm Mon} such that
\begin{enumerate}
\item If ${\bf x}^{\bf i}  > {\bf x}^{\bf j}$, then ${\bf x}^{\bf i} {\bf x}^{\bf k} > {\bf x}^{\bf j}{\bf x}^{\bf k}$ for all ${\bf i}$, ${\bf j}$, ${\bf k} \in {\mathbb N^{n+1}}$.
\item ${\bf x}^{\bf i}  \ge 1$ for all ${\bf i} \in {\mathbb N^{n+1}}$.
\end{enumerate}

We describe two monomial orderings that we will use.
The {\it lexicographic ordering} on $A$ is the monomial ordering $>_{\rm lex}$ such that  ${\bf x}^{\bf i}>_{\rm lex} {\bf x}^{\bf j}$ if the left-most non-zero entry of ${\bf i}-{\bf j}$ is positive.  Let ${\bf u}\in{\mathbb N^{n+1}}$.  We call ${\bf u}$ a {\it weight vector} and define the {\it weight order} associated to $\bf u$ as follows:
$$
{\bf x}^{\bf i}>_{\bf u} {\bf x}^{\bf j}\quad\text{ if } {\bf u}\cdot {\bf i}>{\bf u}\cdot {\bf j}, \text{ or } {\bf u}\cdot {\bf i}={\bf u}\cdot {\bf j} \text{ and }  {\bf x}^{\bf i}>_{\rm lex} {\bf x}^{\bf j}.
$$

If $>$ is a monomial ordering and $F\in \CC[x_0,\hdots,x_n]$ is a nonzero polynomial, 
we let ${\rm TM}(F)$ denote the trailing monomial of $F$ (the smallest monomial appearing in $F$ with a nonzero coefficient).  If $F$ and $G$ are nonzero polynomials, then 
${\rm TM}(FG)={\rm TM}(F){\rm TM}(G)$.  For every nonnegative integer $m$ and subset $T\subset \CC[x_0,\hdots,x_n]$, we let 
$$
T_m=\{P\in T \, |\,\text{$P$ is a homogeneous polynomial of degree $m$ } \}
$$
and let ${\rm TM}(T)=\{{\rm TM}(F): F\in T\}$.  We will use the following key lemma. 
 \begin{lemma}\label{mainlemma}
Let $F_1,F_2\in k[x_0,\hdots,x_n]$ be coprime homogeneous polynomials of degree $d$,
where $k$ is a field.
Let $I_1$, $I_2$, $I_3$ be the principal ideals $I_1=(F_1)$, $I_2=(F_2)$, $I_3=(F_1F_2)$.
Let $m\geq d$ and let $V=(F_1,F_2)_m:=k[x_0,\hdots,x_n]_m\cap (F_1,F_2)$, a   $k$-subspace of the ideal $(F_1,F_2)$.  
Let $V_j=(F_1,F_2)_m\cap I_j$, $j=1,2,3$. Let 
\begin{align*}
B_1&=\{F_1{{\bf x}^{\bf i}}:\, |{\bf i}|=m-d\},\\
B_2&=\{F_2{{\bf x}^{\bf i}}:\, |{\bf i}|=m-d\},\\
B'_1&=\{F_1  {\rm TM}(F_2){\bf x}^{\bf i}:\, |{\bf i}|=m-2d\}.
\end{align*}
Then 
$$
B=(B_1\setminus B_1')\cup B_2
$$
is a basis for $(F_1,F_2)_m$.
Moreover, 
$$
\sum_{s\in B_j} \ord_{x_i} \frac {s}{F_j}=\binom {m+n-d}{n+1}
$$
for $j=1,2$, and
$$
\sum_{s\in B'_1} \ord_{x_i} \frac {s}{F_1}=\binom {m+n-2d}{n+1}+\binom {m+n-2d}{n}\ord_{x_i} {\rm TM}(F_2).
$$
\end{lemma}
\begin{proof}
Since $f_1$ and $f_2$ are coprime, we have $I_1\cap I_2=I_3$ and $V_1\cap V_2=V_3$.  For any finite-dimensional subspace $W\subset k[x_0,\ldots, x_n]$, it is easy to see that $\# {\rm TM}(W)=\dim W$.  Since $V=V_1+V_2$ and $\dim V=\dim V_1+\dim V_2-\dim V_3$, we have $\dim V=\#B_1+\#B_2-\#B_1'$.  Let $V_1'$ be the span of $B_1\setminus B_1'$.  Then $ {\rm TM}(V_1')= {\rm TM}(B_1\setminus B_1')= {\rm TM}(B_1)\setminus  {\rm TM}(B_1')$.  If $p\in V_1'\cap V_2$, $p\neq 0$, then $p\in V_3$ and $ {\rm TM}(p)\in {\rm TM}(V_3)= {\rm TM}(B_1')$, a contradiction.  So $V_1'\cap V_2=0$.  Since $\dim V_1'+\dim V_2=\#B_1-\#B_1'+\#B_2=\dim V$ as well, it follows that $B=(B_1\setminus B_1')\cup B_2$ is a basis for $V$.


As is well-known, the number of monomials of degree $\delta$ in $x_0,\ldots, x_n$ is $\binom{n+\delta}{n}$.  
Then
\begin{align*}
\sum_{|{\bf i}|=m-d} \sum_{i=0}^n\ord_{x_i}{\bf x}^{\bf i}=\sum_{i=0}^n\sum_{|{\bf i}|=m-d} \ord_{x_i}{\bf x}^{\bf i}=(m-d)\binom{m+n-d}{n}.
\end{align*}
By symmetry, we have
\begin{align*}
\sum_{|{\bf i}|=m-d} \ord_{x_i}{\bf x}^{\bf i}=\frac{m-d}{n+1}\binom{m+n-d}{n}=\binom {m+n-d}{n+1}
\end{align*}
for $i=0,\ldots, n$. This yields the first summation formula.  The second summation formula follows similarly.
\end{proof}


\section{Key Theorem}
\label{sKey}
In this section, we prove a fundamental result underlying the proofs of our main theorems.

\begin{theorem}\label{Refinement}
Let $F, G\in\mathbb C[x_0, \cdots,x_n ]$ be coprime homogeneous polynomials of the same degree $d>0$.  Let $I$ be the set of exponents ${\bf i}$ such that ${\bf x}^{\bf i}$ appears with a nonzero coefficient in either $F$ or $G$.  Let $m\ge d$ be a positive integer.   Let $g_0,g_1,\dots,g_n $ be entire functions without common zeros such that the set $\{g_0^{i_0}\cdots g_n^{i_n}: i_0+\cdots+i_n=m\}$ is linearly independent over $\CC$.  Then for any $\epsilon>0$, there exists a positive integer $L$ such that the following holds:
\begin{align*}
& MN_{\rm gcd}(F({\bf g}),G({\bf g}),r) \\
&\le_{\rm exc} c_{m,n,d}  \sum_{i=0}^n N^{(L)}_{ g_i}(0,r)+ \left(\frac{m}{n+1}\binom{m+n}{n}-c_{m,n,d}-M'm\right)\sum_{i=0}^nN_{g_i}(0,r)\\ 
  &+\binom {m+n-2d}{n}N_{\rm gcd}(\{{\bf g}^{\bf i}\}_{{\bf i}\in I},r)+ (M'mn+\epsilon m )T_{\bf g}(r)+O(1),
\end{align*}
where ${\bf g}=(g_0,g_1,\dots,g_n):\CC \to \PP^n$, $c_{m,n,d}=2\binom {m+n-d}{n+1}- \binom {m+n-2d}{n+1}$,
$M=M_{m,n,d}=2\binom {m+n-d}{n}- \binom {m+n-2d}{n}$, and $M'$ is an integer of order $O(m^{n-2})$.
\end{theorem}

\begin{proof}[Proof of Theorem \ref{Refinement}]
Let $\{\phi_1,\hdots,\phi_M\}$ be a basis of the $\CC$-vector space $(F,G)_m$, where $M=\dim (F,G)_m$. For each $z\in \CC$, we construct a basis $B_z$ for $V_m=\CC[x_0,\hdots,x_n]_m/(F,G)_m$ as follows.  
For ${\bf i}=(i_0,\cdots,i_n)\in\mathbb{N}^{n+1}$, we let ${\bf g}(z) ^{\bf i}   =g_0(z)^{i_0}\cdots g_n(z)^{i_n}$.
Choose a monomial ${\bf x}^{{\bf i}_1}\in \CC[x_0,\hdots,x_n]_m$   so that    $ |{\bf g}(z)^{{\bf i}_1}| $ is minimal subject to the condition ${\bf x}^{{\bf i}_1}\notin (F,G).$  Suppose now that ${\bf x}^{{\bf i}_1},\hdots,{\bf x}^{{\bf i}_j}$ have been constructed and are linearly independent modulo $(F,G)_m$, but don't span $\CC[x_0,\hdots,x_n]_m$ modulo $(F,G)_m$.  Then we let ${\bf x}^{{\bf i}_{j+1}}\in \CC[x_0,\hdots,x_n]_m$ be a monomial such that $|{\bf g}(z) ^{{\bf i}_{j+1}}|$ is minimal subject to the condition that 
${\bf x}^{{\bf i}_1},\hdots,{\bf x}^{{\bf i}_{j+1}}$ are linearly independent modulo $(F,G)_m$.  In this way, we construct a basis of $V_m$ with monomial representatives ${\bf x}^{{\bf i}_1},\hdots,{\bf x}^{{\bf i}_{M'}}$, where $M'=\dim V_m$.
Let $I_z=\{ {\bf i}_1,\hdots,{\bf i}_{M'}\}$.  Then for each ${\bf i}$, $|{\bf i}|= m$, we have
$$
{\bf x}^{\bf i}+\sum_{j=1}^{M'}c_{{\bf i},j}{\bf x}^{{\bf i}_j}\in (F,G)_m
$$
for some choice of coefficients $c_{{\bf i},j}\in\CC$.  Then for each such ${\bf i}$ there is a linear forms $L_{z,{\bf i}}$ over $\CC$ such that 
$$
L_{z,{\bf i}}(\phi_1,\hdots,\phi_M)={\bf x}^{\bf i}+\sum_{j=1}^{M'}c_{{\bf i},j}{\bf x}^{{\bf i}_j}.
$$
We note that since there are only finitely many choices of a monomial basis of $V_m$, there are only finitely many choices of $c_{{\bf i},j}$, even as $z$ runs through all of $\CC$. 
Note also that $\{L_{z,{\bf i}}(\phi_1,\hdots,\phi_M) \,|\, |{\bf i}|= m, {\bf i}\notin I_z\}$ is a basis for $(F,G)_m$.
From the definition of ${\bf x}^{{\bf i}_1},\hdots,{\bf x}^{{\bf i}_{M'}}$, we have the key inequality
\begin{align}\label{keyinequality2}
\log |L_{z,{\bf i}}(\Phi ({\bf g}(z))|\le \log |{\bf g}(z)^{{\bf i} }|+C,
\end{align}
where $\Phi ({\bf g}(z)) =(\phi_1({\bf g}(z)),\hdots,\phi_M({\bf g}(z)))$, and the constant $C$ is independent of $z$ as the choice of $c_{{\bf i},j}$ is finite.
The  map $ (\phi_1({\bf g} ),\hdots,\phi_M({\bf g} ) )$ may not be a reduced presentation of $\Phi ({\bf g} )$.  Let $h$ be an entire function such that $F({\bf g} )/h$ and $G({\bf g})/h$ are entire and have no common zeros, i.e., $h$ is a gcd of $F({\bf g} ) $ and $G({\bf g} )$.  Let $\psi_i:=\phi_i({\bf g})/h $, $1\le i\le M$.  As  $\phi_i\in (F,G)$,  the function $\psi_i$ is entire for $1\le i\le M$.  Moreover, since $FX_i^{m-d}, GX_i^{m-d}\in (F,G)_m$, $0\leq i\leq n$, and $g_0,\ldots, g_n$ have no common zero, the functions $\psi_i$, $1\le i\le M$, have no common zero.  Hence  $\Psi:=(\psi_1 ,\hdots,\psi_M )$ is a reduced form of $\Phi ({\bf g})$.

Applying Theorem \ref{gsmt}, the second main theorem, to the map $\Psi ({\bf g}(z))$ with the choice of linear forms $L_{z,{\bf i}}$, $|{\bf i}|= m$, ${\bf i}\notin I_z$, $z=re^{i\theta}\in \mathbb{C}$, we get for any $\epsilon>0$,
\begin{align}\label{useSMT2}
&\int_0^{2\pi}   \sum_{|{\bf i}|= m, {\bf i}\notin I_z}(-\log |L_{z,{\bf i}}(\Phi ({\bf g}(re^{i\theta}))|+\log\|\Phi ({\bf g}(re^{i\theta}))\|  )\frac{d\theta}{2\pi} +N_{W(\Psi)}(0,r)\cr
&\le_{\rm exc} (M+\epsilon) T_{\Phi ({\bf g)}}(r).
\end{align}
 
 Next, we will derive a lower bound for the left hand side of (\ref{useSMT2}).  By definition of the characteristic function, Lemma \ref{Jensen} and the choice of $h$, we have
\begin{align*} 
\int_0^{2\pi}   \log\|\Phi ({\bf g}(re^{i\theta})) )\|  \frac{d\theta}{2\pi} 
&= \int_0^{2\pi}\log \max\{|\psi_1(re^{i\theta}),\hdots, \psi_M(re^{i\theta})|\}+\log| h(re^{i\theta})|\frac{d\theta}{2\pi}\cr
&= T_{\Phi ({\bf g)}}(r)+ N_{\rm gcd}(F({\bf g}),G({\bf g}),r)+O(1).
\end{align*}
On the other hand, since $\phi_i\in \CC[x_0,\hdots,x_n]_m$,
$$
\log|\phi_i({\bf g}(z))|\le m\log\max\{|g_0(z)|,\hdots, |g_n(z)|\}+O(1),
$$
and hence
\begin{align*}
 \int_0^{2\pi}  \log\|\Phi ({\bf g}(re^{i\theta}))\|  \frac{d\theta}{2\pi}
 \le m  T_{\bf g}(r)+O(1).
\end{align*}
In conclusion, we have
\begin{align}\label{height}
N_{\rm gcd}(F({\bf g}),G({\bf g}),r)+T_{\Phi ({\bf g)}}(r)=\int_0^{2\pi}   \log\|\Phi ({\bf g}(re^{i\theta}))\|  \frac{d\theta}{2\pi}+O(1)  \le m T_{\bf g}(r)+O(1).
\end{align} 
 
To estimate the first term of (\ref{useSMT2}), we use the key inequality (\ref{keyinequality2}) to derive
\begin{align}\label{estimateL}
  -\sum_{|{\bf i}|= m, {\bf i}\notin I_z}& \log |L_{z,{\bf i}}(\Phi ({\bf g}(z) )| +c
   \ge  -\sum_{|{\bf i}|= m, {\bf i}\notin I_z} \log | {\bf g}(z)^{{\bf i} }|\cr
  &\ge   -\sum_{|{\bf i}|= m } \log | {\bf g}(z)^{{\bf i} }|+   \sum_{|{\bf i}|= m, {\bf i}\in I_z} \log | {\bf g}(z)^{{\bf i} }|\cr
  &\ge -\sum_{|{\bf i}|= m } \log | {\bf g}(z)^{{\bf i} }|+  M'm\log\min\{ | g_0(z)|,\hdots, | g_n(z)|\}
   \end{align}
   where $c$ is a constant independent of $z$.
  Note  that
\begin{align*}
\sum_{|{\bf i}|= m}\log|{\bf g}(z)^{\bf i}|=\frac{m}{n+1}\binom{m+n}{n}\sum_{i=0}^n\log|g_i(z)|,
\end{align*}
and 
\begin{align*}
\log\min\{ | g_0(z)|,\hdots, | g_n(z)|\}
&\ge \sum_{i=0}^n\log |g_i(z)|-n\max\{\log|g_0(z)|,\hdots,\log|g_n(z)|\}.
\end{align*}
By Lemma \ref{Jensen}, Jensen's formula, the integration of (\ref{estimateL}) from 0 to $2\pi$ over 
$d\theta$ gives
 \begin{align}\label{Li2}
\int_0^{2\pi}  &\sum_{|{\bf i}|= m, {\bf i}\notin I_z} -\log |L_{z,{\bf i}}(\Phi ({\bf g}(re^{i\theta}) ))|  \frac{d\theta}{2\pi}\cr
&\ge -\left(\frac{m}{n+1}\binom{m+n}{n}-M'm\right)\sum_{i=0}^nN_{g_i}(0,r)-M'mnT_{\bf g}(r)+O(1).
\end{align}

It then follows from  (\ref{useSMT2}), (\ref{height}), and (\ref{Li2}) that
\begin{align}\label{maininequality}
&MN_{\rm gcd}(F({\bf g}),G({\bf g}),r) \cr
 &\le_{\rm exc} \left(\frac{m}{n+1}\binom{m+n}{n}-M'm\right)\sum_{i=0}^nN_{g_i}(0,r)-N_{W(\Psi)}(0,r)+(M'mn+\epsilon m) T_{\bf g}(r)+O(1).
\end{align} 

By direction calculation, we find that $M'=\binom{m+n}{n}-M=O(m^{n-2})$.  Alternatively, since $F$ and $G$ are coprime, the ideal $(F,G)$ defines a closed subset of $\mathbb P^n$ of codimension at least 2, and it follows from the theory of Hilbert functions and Hilbert polynomials that $M'=O(m^{n-2})$.

It's clear from (\ref{maininequality}) that it suffices to show that there exists a large integer $L$ (to be determined later) such that 
  \begin{equation}
    \label{truncation2}
        \begin{aligned}
       c_{m,n,d}\sum_{i=0}^n N_{g_i}(0,r) -\binom {m+n-2d}{n}N_{\rm gcd}(\{{\bf g}^{\bf i}\}_{{\bf i}\in I},r)-N_{W(\Psi)}(0,r)\le  c_{m,n,d}\sum_{i=0}^n N^{(L)}_{g_i}(0,r).
        \end{aligned}
    \end{equation} 
The above inequality can be deduced from the inequality
  \begin{equation}
    \label{localeq2}
        \begin{aligned}
       c_{m,n,d}\sum_{i=0}^n  v_{z}^+ ( g_i)- \binom {m+n-2d}{n}\min_{{\bf i}\in I}v_z^+({\bf g}^{\bf i}) -v_{z}^+ ( W(\Psi))\le  c_{m,n,d}\sum_{i=0}^n \min\{L,v_{z}^+  (g_i)\},
        \end{aligned}
    \end{equation} 
 for all $z\in\CC$.  The inequality holds trivially if $v_{z}^+(g_i)\le  L$ for each $0\le i\le n$.
Therefore, we only need to consider the case where $v_{z}^+(g_i)> L$ for some $0\le i\le n$.

For $z\in\CC$, we define a monomial ordering $>_{{\bf g}(z)}$ on $A=\CC[x_0,\cdots,x_n]$ using the weight vector ${\bf u}=(v_z(g_0),\ldots, v_z(g_n))$.  Let
  \begin{align*}
B_1&=\{F_1{{\bf x}^{\bf i}}:\, |{\bf i}|=m-d\},\\
B_2&=\{F_2{{\bf x}^{\bf i}}:\, |{\bf i}|=m-d\},\\
B'_1&=\{F_1{\rm TM}_{{\bf g}(z)}(F_2){\bf x}^{\bf i}:\, |{\bf i}|=m-2d\},
\end{align*}
where $\{F_1,F_2\}=\{F,G\}$ and ${\rm TM}_{{\bf g}(z)}(F_2)\le {\rm TM}_{{\bf g}(z)}(F_1)$.
By Lemma \ref{mainlemma} 
$$
B=(B_1\setminus B_1')\cup B_2
$$
is a basis for $(F,G)_m$. 
Write $B=\{\beta_1,\hdots,\beta_M\}$.
Let $\eta_j= \beta_j({\bf g})/h$ (note that the $\beta_j$ depend on $z$.)  From the definition of $>_{{\bf g}(z)}$ and $F_2$, it follows that 
\begin{align*}
v_z^+({\rm TM}_{{\bf g}(z)}(F_2))=\min_{{\bf i}\in I}v_z^+({\bf g}^{\bf i}),
\end{align*}
where $I$ is the set of exponents ${\bf i}$ such that ${\bf x}^{\bf i}$ appears with a nonzero coefficient in either $F$ or $G$.
Then by the second part of Lemma \ref{mainlemma}, we have  for each $z\in \CC$,
\begin{align}\label{BS} 
\sum_{j=1}^M  v_{z}^+ (\eta_j)  
&\geq  \sum_{i=0}^n \left(
\sum_{s\in B_1} \ord_{x_i} \frac {s}{F_1}+\sum_{s\in B_2} \ord_{x_i} \frac {s}{F_2}-\sum_{s\in B'_1} \ord_{x_i} \frac {s}{F_1}\right)
\cdot v_{z}^+ (g_i)  \cr
&\geq  c_{m,n,d}\sum_{i=0}^{n}v_{z}^+ (g_i)-\binom {m+n-2d}{n}\min_{{\bf i}\in I}v_z^+({\bf g}^{\bf i}),
\end{align} 
where 
$$
c_{m,n,d}=2\binom {m+n-d}{n+1}- \binom {m+n-2d}{n+1}.
$$ 
On the other hand, from the basic properties of Wronskians, we have
 \begin{equation}
    \label{ordW2}
        \begin{aligned}
     v_{z}^+(W(\Psi))&\ge \sum_{j=1}^M  v_{z}^+ (\eta_j)  -\frac12M(M-1).        
      \end{aligned}
 \end{equation}  
Combining (\ref{BS}) and (\ref{ordW2}), we obtain that
\begin{equation*}
        \begin{aligned}
       c_{m,n,d}\sum_{i=0}^{n}v_{z}^+ (g_i)-\binom {m+n-2d}{n}\min_{{\bf i}\in I}v_z^+({\bf g}^{\bf i})- v_{z}^+(W(\Psi))
        \le  \frac 12  M(M-1) .        
        \end{aligned}
    \end{equation*}
Let $L=\frac 12  M(M-1)c_{m,n,d}^{-1} $.   
The assumption that $v_{z}^+(g_i)> L$ for some $0\le i\le n$ implies that 
\begin{align*}
\frac 12  M(M-1)=c_{m,n,d}L\le c_{m,n,d}\sum_{i=0}^n \min\{L,v_{z}^+ (g_i)\}.
\end{align*}
\end{proof}

 
\section{Proof of the Main Theorems}
\label{sMain}
We  recall Borel's lemma. (See \cite[Theorem A.3.3.2]{ru2001nevanlinna}).
\begin{lemma}[Borel's Lemma]\label{Borel}
Let $f_0 , \hdots,  f_{n+1}$ be entire functions without zeros, satisfying
$$
f_0+\hdots+f_n+f_{n+1}=0.
$$
Define an equivalence relation $i\sim j$ if $f_i/f_j$ is constant.  Then for each equivalence class $S$ we have 
$$
\sum_{i\in S}f_i=0.
$$  
\end{lemma}

We also recall the following result of Green \cite{green1975} (See \cite[Chapter VII, Theorem 4.1]{La}).
\begin{lemma}\label{greenlemma}
Let $f_0,\hdots,f_n$ be  entire functions with no common zeros satisfying
$$
f_0^k+\cdots+f_n^k=0.
$$
Suppose that none of the $f_i$ are identically $0$.  Define an equivalence relation $i\sim j$ if $f_i/f_j$ is constant. 
If $k\ge n^2$, then for each equivalence class $S$ we have 
$$
\sum_{i\in S}f_i^k=0.
$$  
\end{lemma}

We will also make use of the following result on proximity functions.

\begin{theorem}\label{Theorem10.3}
Let $G\in\mathbb C[x_1, \cdots,x_n ]$ be a polynomial that does not vanish at the origin $(0,\hdots,0)$. 
Suppose that $g_1,\dots,g_n $ are entire functions such that $g_1^{i_1} \cdots g_n^{i_n} \notin\CC$ for any 
index set $(i_1,\dots,i_n)\in\mathbb Z^n\setminus \{(0,\dots,0)\}$.  For all $\epsilon>0$,  
\begin{enumerate}
\item there exists a positive integer $k_0$ such that for all $k\ge k_0$,
$$
m_{G(g_1^k,\hdots,g_n^k)}(0,r)\le_{\rm exc}\epsilon \max_{1\le i\le n}\{T_{g_i^k}(r)\};
$$
\item if in addition each $g_i$, $1\le i\le n$, has no zero, then 
$$
m_{G(g_1,\hdots,g_n)}(0,r)\le_{\rm exc}\epsilon \max_{1\le i\le n}\{T_{g_i}(r)\}.
$$
\end{enumerate}
\end{theorem}
 \begin{proof} 
By  Lemma \ref{greenlemma}, the set $\{{\bf g}^{k{\bf i}}:=g_1^{ki_1}\cdots g_n^{ki_n}\,|\, {\bf i}=(i_1,\hdots,i_n)\in \mathbb{N}^n, |{\bf i}|\le d\}$ 
is linearly independent for $k\ge \binom {d+n}{n}^2$.  Since $G(0,\hdots,0)\ne 0$, $G$ must have a non-zero constant term, and by  arranging the index set in some order, we may write  
$$
G=a_{{\bf i}(0)}+\sum_{j=1}^{\ell}a_{{\bf i}(j)}{\bf x}^{{\bf i}(j)},
$$      
where $a_{{\bf i}(j)}\ne 0$ for $0\le j\le \ell$.  Then we have  
$$
G(g_1^k,\hdots,g_n^k)=a_{{\bf i}(0)}+\sum_{j=1}^{\ell}a_{{\bf i}(j)}{\bf g}^{k{\bf i}(j)}.
$$  

We apply  Theorem \ref{tsmt}, Cartan's truncated second main theorem, to the holomorphic map 
$$
\frak g(k):=(1,{\bf g}^{k{\bf i}(1)},\hdots,{\bf g}^{k{\bf i}(\ell)}):\CC\to \mathbb P^{\ell}
$$ associated with the above expression ($k\ge \binom {d+n}{n}^2$), and with the set of hyperplanes given by the coordinate hyperplanes of $\PP^{\ell}$ and the one defined by $\sum_{i=0}^{\ell}a_{{\bf i}(j)}X_j$. Then for any $\epsilon>0$, we have
\begin{align}\label{usingtsmt}
(1-\frac \epsilon2)T_{\frak g(k)}\le_{\operatorname{exc}} N^{(\ell)}_{G(g_1^k,\hdots,g_n^k)}(0,r)+\sum_{j=1}^{\ell} N^{(\ell)}_{{\bf g}^{k{\bf i}(j)}}(0,r).
\end{align}
We note that
$$
N^{(\ell)}_{{\bf g}^{k{\bf i}(j)}}(0,r)\le  \frac{\ell}k N_{{\bf g}^{k{\bf i}(j)}}(0,r)\le  \frac{\ell}k T_{{\bf g}^{k{\bf i}(j)}}(r)\le \frac{\ell}kT_{\frak g(k)}(r),
$$
where the last inequality is due to the definition of characteristic functions and that $\frak g(k)=(1,{\bf g}^{k{\bf i}(1)},\hdots,{\bf g}^{k{\bf i}(\ell)})$.  Then for $k>\max\{\frac{2\ell^2}{\epsilon}, \binom {d+n}{n}^2\}$,
\begin{align}\label{usetsmta}
(1- \epsilon )T_{\frak g(k)}(r)\le_{\operatorname{exc}} N_{G(g_1^k,\hdots,g_n^k)}(0,r)=T_{G(g_1^k,\hdots,g_n^k)}(r)-m_{G(g_1^k,\hdots,g_n^k)}(0,r)+O(1).
\end{align}
Since  $G(g_1^k,\hdots,g_n^k)$ is an entire function, 
$$
T_{G(g_1^k,\hdots,g_n^k)}(r)=m_{G(g_1^k,\hdots,g_n^k)}(\infty,r)\le T_{\frak g(k)}(r)+O(1).
$$
Consequently,
$$
m_{G(g_1^k,\hdots,g_n^k)}(0,r)\le_{\operatorname{exc}}  \epsilon T_{\frak g(k)}(r)\le_{\operatorname{exc}} d \epsilon T_ {(1,g_1^k,\cdots,g_n^k)}(r)\le dn \epsilon \max_{1\le i\le n}\{T_{g_i^k}(r)\}.
$$

When  the $g_i$ are entire functions without zeros, we may assume  that the set $\{{\bf g}^{ {\bf i}}:=g_1^{ i_1}\cdots g_n^{ i_n}\,|\, {\bf i}=(i_1,\hdots,i_n)\in \mathbb{N}^n, |{\bf i}|\le d\}$ 
is linearly independent by Lemma \ref{Borel}.  Then we can repeat the previous argument  for $k=1$ with the additional condition $N_{g_i}(0,r)=O(1)$, $1\le i\le n$, to conclude the proof.
 \end{proof}

\begin{proof}[Proof of Theorem \ref{RefinementII} and Theorem \ref{gcdunit}]
We first claim that  $g_1,\dots,g_n$ are algebraically independent over $\CC$ under the assumption that $g_1^{i_1} \cdots g_n^{i_n} \notin\CC$ for any 
index set $(i_1,\dots,i_n)\in\mathbb Z^n\setminus \{(0,\dots,0)\}$.
If not, then they satisfy a non-trivial $\CC$-linear relation, say
\begin{align}\label{relation1}
\sum_{\bf i} a_{\bf i} g_1^{i_1}\cdots g_n^{i_n}=0,
\end{align}
where the sum is over finitely many index sets ${\bf i}=(i_1,\dots,i_n)$ and $a_{\bf i}\in\CC^*$.  We may further assume that no proper sub-sum
 of the left hand side of (\ref{relation1})
is zero.  Since $g_1,\ldots, g_n$ are entire functions without zeros,   Borel's Lemma   implies that we have a pair of indices ${\bf i}=(i_1,\dots,i_n)\ne {\bf j}=(j_1,\dots,j_n)$ such that $a_{\bf i} g_1^{i_1}\cdots g_n^{i_n}$ is a constant multiple of  some $a_{\bf j}g_1^{j_1}\cdots g_n^{j_n}$  appearing on the left hand side of \eqref{relation1}, and hence $g_1^{i_1-j_1}\cdots g_n^{i_n-j_n}\in\CC^*$, contradicting our assumption.

We prove part \ref{refa} of both theorems first.  Consider $(\mathbb{C}^*)^n\subset \mathbb{P}^n$, where we identify $(x_1,\ldots, x_n)\in(\mathbb{C}^*)^n$ with $[1:x_1:\cdots :x_n]$, and let $\overline{Y}$ be the Zariski closure of $Y$ in $\mathbb{P}^n$.  Then $N_{\bf g}(Y,r)=N_{\bf g}(\overline{Y},r)$, and we may assume that $Y$ is a closed subscheme of $\mathbb{P}^n$ (of codimension at least $2$).

 Let $I\subset \CC[x_0,\hdots,x_n]$ be a homogeneous ideal associated to $Y$.  We can find homogeneous polynomials $\tilde F,\tilde G\in  I$ of the same degree $d$ such that $\tilde F$ and $\tilde G$ are coprime.
 Let $Y'$ be the closed subscheme defined by the ideal $(\tilde F,\tilde G)$.  Then
 $$
  N_{{\bf g}}(Y, r)  \leq  N_{{\bf g}}(Y', r)   
 $$
 for all $r>0$.
 Therefore, we may assume that $Y$ is defined by  the ideal $(\tilde F,\tilde G)$ after replacing $Y$ by $Y'$.
 Then  for any $\epsilon>0$, Theorem \ref{Refinement} for any (large) $m$ implies that
 \begin{align}\label{estimate}
 MN_{{\bf g}}(Y, r)=MN_{\rm gcd}(\tilde F({\bf g}),\tilde G({\bf g}),r) \le_{\rm exc}  (M'mn+\epsilon m ) T_{\bf g}(r)+O(1)
\end{align} 
as the $g_i$ are entire functions with no zeros.
Let $\epsilon'>0$.  Since $M'=O(m^{n-2})$ and $M=\frac{m^n}{n!}+ O(m^{n-1})$,    choosing $m$ large enough, depending only on $\epsilon'$, (\ref{estimate}) implies that
\begin{align}\label{firstcount}
N_{\rm gcd}(\tilde F({\bf g}),\tilde G({\bf g}),r)\le_{\rm exc}  \epsilon' T_{\bf g}(r).
\end{align} 
This completes the proof of Theorem \ref{RefinementII}\ref{refa}.

To show Theorem \ref{gcdunit}\ref{gcda}, we note that we may assume $\deg F=\deg G=d$  since
$$
N_{\gcd}(F(g_1,\hdots,g_n), G(g_1,\hdots,g_n),r)\le  N_{\gcd}(F^{e}(g_1,\hdots,g_n), G^{h}(g_1,\hdots,g_n),r),
$$
where $e=\deg G$ and $h=\deg F$.
Then Theorem \ref{gcdunit}\ref{gcda} follows from (\ref{firstcount}) by taking
\begin{align*}
\tilde F (x_0,\dots,x_n) =x_0^d  F\left(\frac{x_1}{x_0},\dots,\frac{x_n}{x_0}\right),\  
\tilde G(x_0,\dots,x_n) =x_0^d G\left(\frac{x_1}{x_0},\dots,\frac{x_n}{x_0}\right),
\end{align*}
and using Lemma \ref{gineq}.

Suppose in addition that, say, $G$ does not vanish at the origin $(0,\ldots, 0)$.  It is immediate from the definitions that
\begin{align*}
m_{\gcd}(F(g_1,\hdots,g_n), G(g_1,\hdots,g_n),r)\leq m_{G(g_1,\hdots,g_n)}(0,r),
\end{align*}
and so Theorem \ref{Theorem10.3} implies that
\begin{align*}
m_{\gcd}(F(g_1,\hdots,g_n), G(g_1,\hdots,g_n),r)\le_{\rm exc}  \epsilon\max_{1\le i\le n}\{ T_{g_i}(r)\}.
\end{align*}
Combined with part \ref{gcda} above and Lemma \ref{gcdfirstmain}, we obtain Theorem \ref{gcdunit}\ref{gcdb}.

It remains to prove Theorem \ref{RefinementII}\ref{refb}. Let $X$ be a nonsingular projective toric compactification of $(\CC^*)^n$.  Let $\overline{Y}$ be the Zariski closure of $Y$ in $X$, and suppose that $\overline{Y}$ is in general position with the boundary of $(\CC^*)^n$ in $X$.  Since $N_{\bf g}(Y,r)=N_{\bf g}(\overline{Y},r)$, by part \ref{refa} it suffices to show that $m_{\bf g}(\overline{Y},r)\le_{\rm exc}  \epsilon T_{\bf g}(r).$  In fact, we will show the stronger statement with $\overline{Y}$ replaced by an effective divisor $D$ on $X$ in general position with the boundary of $(\CC^*)^n$ in $X$. We follow the proof of \cite[Theorem 4.4]{levin2017}, which shows that there exist embeddings $\phi_i:(\CC^*)^n\to \mathbb{A}^n$ (which extend an automorphism of $(\CC^*)^n$) and polynomials $p_i\in \CC[x_1,\ldots, x_n]$ nonvanishing at the origin, $i=1,\ldots, t$, such that for every $P\in (\CC^*)^n\subset X(\CC)$, there exists $i\in \{1,\ldots, t\}$ satisfying
\begin{align*}
\lambda_D(P)= -\log |p_i(\phi_i(P))|+O(1).
\end{align*}
Thus,
\begin{align*}
\lambda_D(P)\leq -\sum_{i=1}^t\log^- |p_i(\phi_i(P))|+O(1)
\end{align*}
for all $P\in (\CC^*)^n\subset X(\CC)$.  By Theorem \ref{Theorem10.3}, this implies that for any $\epsilon>0$,
\begin{align*}
m_{\bf g}(D,r)&\leq \sum_{i=1}^t m_{p_i(\phi_i(g_1,\ldots, g_n))}(0,r)+O(1)\\
&\leq_{\rm exc} \epsilon T_{\bf g}(r)+O(1),
\end{align*}
completing the proof.

\end{proof}

\begin{proof}[Proof of Theorem \ref{gcdPn}]
By replacing $F$ and $G$ by suitable linear combinations of $F$ and $G$, we may assume that $F$ and $G$ have the same degree $d$,  and that neither $F$ nor $G$ vanishes at the origin.  We now consider the two related homogeneous polynomials
\begin{align*}
F_1(x_0,x_1,\dots, x_{n}) &=x_{0}^{d}  F\left(\frac{x_1}{x_0},\dots,\frac{x_n}{x_0}\right),\\  
G_1(x_0,x_1,\dots, x_{n}) &=x_{0}^{d}  G\left(\frac{x_1}{x_0},\dots,\frac{x_n}{x_0}\right).
\end{align*}
 Since $F$ and $G$ are coprime, it follows easily that $F_1$ and $G_1$ are coprime.
Let $g_1,\dots,g_n$ be meromorphic functions.
Then there exists an entire function $h_0$ such that $h_0$ and $h_i:=g_i\cdot h_0$, $1\le i\le n$,  are entire functions without a common zero.  Therefore, $(h_0,h_1,\hdots,h_n)$ is a reduced form of  the holomorphic map ${\bf g}:=[1:g_1:\cdots:g_n]:\CC\to \PP^n$.   
Assume that $g_1^{ j_1}\dots g_n^{ j_n}$ is not constant for any $(j_1,\dots,j_n)\ne(0,\dots,0)\in\mathbb Z^n$.   
Let
\begin{align}\label{lowerk}
k\ge \binom {m+n}{n}^2.
\end{align}
Then the set $\{(h_0^k)^{i_1}\cdots (h_n^k)^{i_n}: i_0+\cdots+i_n= m\}$ is linearly independent over $\CC$ by Lemma \ref{greenlemma}.
Let $ {\bf h}^k =(h_0^k,\cdots,h_n^k): \mathbb C \to \PP^n$.  Let $I$ be the set of exponents ${\bf i}$ such that ${\bf x}^{\bf i}$ appears with a nonzero coefficient in either $F_1$ or $G_1$.  Note that

\begin{align*}
F_1({\bf h}^k) = h_0^{kd}F(g_1^k,\dots,g_n^k),\text{ and }
G_1({\bf h}^k) =  h_0^{kd}G(g_1^k,\dots,g_n^k),
\end{align*}
and so for any $z\in \mathbb{C}$,
\begin{align*}
v_z(F_1({\bf h}^k)) = v_z(h_0^{kd})+v_z(F(g_1^k,\dots,g_n^k))\geq \min_{{\bf i}\in I}v_z({\bf h}^{k{\bf i}})+v_z(F(g_1^k,\dots,g_n^k)),\\
v_z(G_1({\bf h}^k)) = v_z(h_0^{kd})+v_z(G(g_1^k,\dots,g_n^k))\geq \min_{{\bf i}\in I}v_z({\bf h}^{k{\bf i}})+v_z(G(g_1^k,\dots,g_n^k)),
\end{align*}
since $x_0^d$ is a monomial appearing nontrivially in $F_1$ and $G_1$ as $F$ and $G$ don't vanish at the origin.  Then
\begin{align*}
v_z(F_1({\bf h}^k))- \min_{{\bf i}\in I}v_z({\bf h}^{k{\bf i}})\geq v_z^+(F(g_1^k,\dots,g_n^k)),\\
v_z(G_1({\bf h}^k)) -\min_{{\bf i}\in I}v_z({\bf h}^{k{\bf i}})\geq v_z^+(G(g_1^k,\dots,g_n^k)),
\end{align*}
and it follows that
\begin{align*}
N_{\rm gcd}(F_1({\bf h}^k),G_1({\bf h}^k),r)- N_{\rm gcd}(\{{\bf h}^{k{\bf i}}\}_{{\bf i}\in I},r)\geq N_{\gcd}(F(g_1^k,\dots,g_n^k), G(g_1^k,\dots,g_n^k),r).
\end{align*}
Then by Theorem \ref{Refinement}, for any $\epsilon>0$ there exists a positive integer $L$ such that for all positive integers $k$ satisfying (\ref{lowerk}), we have  the following:
\begin{align*}
  M&N_{\gcd}(F(g_1^k,\dots,g_n^k), G(g_1^k,\dots,g_n^k),r) 
   \le M(N_{\rm gcd}(F_1({\bf h}^k),G_1({\bf h}^k),r)- N_{\rm gcd}(\{{\bf h}^{k{\bf i}}\}_{{\bf i}\in I},r))\\
&\le_{\rm exc} c_{m,n,d}  \sum_{i=0}^n  N^{(L)}_{ h_i^k}(0,r)+ \left(\frac{m}{n+1}\binom{m+n}{n}-c_{m,n,d}-M'm\right)\sum_{i=0}^nN_{h_i^k}(0,r)\\
   &+  \left(\binom {m+n-2d}{n}-M\right)N_{\rm gcd}(\{{\bf h}^{k{\bf i}}\}_{{\bf i}\in I},r) + (M'mn+\epsilon m)T_{{\bf h}^k}(r)+O(1),
\end{align*}
where $c_{m,n,d}=2\binom {m+n-d}{n+1}- \binom {m+n-2d}{n+1}$,
$M=2\binom {m+n-d}{n}- \binom {m+n-2d}{n}$, and $M'$ is an integer of order $O(m^{n-2})$.
Elementary computations give that 
\begin{align}
\binom {m+n}{n} &=\frac{m^n}{n!}+\frac{(n+1)m^{n-1}}{2(n-1)!}+O(m^{n-2}),\cr
 c_{m,n,d}&=\frac{m^{n+1}}{(n+1)!}+\frac{m^n}{2(n-1)!}+O(m^{n-1}),\cr
 M&= \frac{m^n}{n!}+ O(m^{n-1}). 
\end{align}
Then 
\begin{align*}
\frac{m}{n+1}\binom{m+n}{n}-c_{m,n,d}&=O(m^{n-1}),\\
\binom {m+n-2d}{n}-M&=O(m^{n-1}).
\end{align*}
Furthermore,

\begin{align*}
\sum_{i=0}^n N^{(L)}_{ h_i^k}(0,r)\leq \frac{L}{k}\sum_{i=0}^n N_{ h_i^k}(0,r),
\end{align*}
and
\begin{align*}
\sum_{i=0}^n N_{ h_i^k}(0,r)=\sum_{i=0}^n N_{{\bf h}^k}(H_i,r)\le (n+1)T_{{\bf h}^k}(r)=(n+1)T_{{\bf g}^k}(r),
\end{align*} 
where $H_i$ is the coordinate hyperplane defined by $x_i=0$.
Then we find that after choosing $m$ sufficiently large, for all sufficiently large $k$ (depending on $m$),
\begin{align}\label{gcdkcounting}
N_{\gcd}(F(g_1^k,\dots,g_n^k), G(g_1^k,\dots,g_n^k),r)\le_{\rm exc} \epsilon T_{{\bf g}^k}(r).
\end{align}
This concludes the proof of Theorem \ref{gcdPn} \ref{kgcda}.
If in addition $g_1,\hdots,g_n$ are entire functions and, say, $G$ does not vanish at the origin, then 
\begin{align*}
m_{\gcd}(F(g_1^k,\hdots,g_n^k), G(g_1^k,\hdots,g_n^k),r)\leq m_{G(g_1^k,\hdots,g_n^k)}(0,r),
\end{align*}
and so Theorem \ref{Theorem10.3} implies that
\begin{align*}
m_{\gcd}(F(g_1^k,\hdots,g_n^k), G(g_1^k,\hdots,g_n^k),r)\le_{\rm exc}  \epsilon\max_{1\le i\le n}\{ T_{g_i^k}(r)\}
\end{align*}
for $k$ large enough.  Together with (\ref{gcdkcounting}), we reach the conclusion of  \ref{kgcdb}.
\end{proof}
 
\begin{proof}[Proof of Corollary \ref{fgagcd} ]
If $f^ig^j\notin\CC$ for any $(i,j)\ne(0,0)\in \mathbb Z^2$, then  then the assertion follows from Theorem  \ref{gcdPn} for the meromorphic functions $f$ and $g$ with the polynomials
$F=x_1-1$ and $G=x_2-1$.
Suppose that $f^ig^j=c\in \CC$ for some $(i,j)\in \mathbb Z^2\setminus \{(0,0)\}$.
If $f^k-1$ and $g^k-1$ have no common zero for all $k$ sufficiently large, then  the gcd inequality (\ref{asymgcdcomplex}) holds trivially.  Otherwise, we may find $z_0\in \CC$ such that $f^k(z_0)=g^k(z_0)=1$ for some $k$.  This implies that $c^k=1$ and hence $f^{ik}g^{jk}=1$, contradicting the assumption that 
$f$ and $g$ are multiplicatively independent. 
\end{proof}
The proof of Corollary \ref{fgunit} is similar (with $k=1$) to the proof of Corollary \ref{fgagcd}, and so we omit it.

\begin{proof}[Proof of Corollary \ref{divisibility}]
We prove part \ref{diva}.  The proof of part \ref{divb} is similar.

Suppose that  $g_1^{i_1} \cdots g_n^{i_n} \notin\CC$ for any 
index set $(i_1,\dots,i_n)\in\mathbb Z^{n+1}\setminus \{(0,\dots,0)\}$ and that $F(g_1^k,\hdots,g_n^k)/G(g_1^k,\hdots,g_n^k)$ is an entire function.  Then 
$$
N_{G(g_1^k,\hdots,g_n^k)}(0,r)=N_{\gcd}(F(g_1^k,\hdots,g_n^k), G(g_1^k,\hdots,g_n^k),r),
$$
and by Theorem \ref{gcdPn} \ref{kgcda}, for any $\epsilon>0$, there exists a positive integer $k_0$ such that 
\begin{align}\label{countingG}
N_{G(g_1^k,\hdots,g_n^k)}(0,r) \le_{\rm exc}  \epsilon\max_{1\le i\le n}\{ T_{g_i^k}(r)\},
\end{align}
if $k\ge k_0$.
On the other hand, using the same notation as in the proof of Theorem \ref{Theorem10.3} and recalling the first part of 
(\ref{usetsmta}),
\begin{align}
(1- \epsilon )T_{\frak g(k)}(r)\le_{\operatorname{exc}} N_{G(g_1^k,\hdots,g_n^k)}(0,r)
\end{align}
for $k>k_1:=\max\{\frac{2\ell^2}{\epsilon}, \binom {d+n}{n}^2\}$. 
Together with (\ref{countingG}), we get
$$
T_{\frak g(k)}(r)\le_{\rm exc}  2\epsilon\max_{1\le i\le n}\{ T_{g_i^k}(r)\}+O(1)
$$
for $k\ge \max\{k_0,k_1\}$. 
Hence,
$$
T_{\frak g_G }(r)\le_{\rm exc}  2\epsilon\max_{1\le i\le n}\{ T_{g_i}(r)\}+O(1),
$$
contradicting the assumption that $T_{\frak g_G}(r)\asymp  \max_{1\le i\le n}\{ T_{g_i}(r)\}$.
\end{proof}
\section{GCD with Moving Targets}\label{moving}  
\subsection{Statement of the main results}
Let ${\bf g}$ be a holomorphic map from $\CC$ to  $\PP^n$.   Let $K_{\bf g}$ be the set containing all meromorphic functions $a$ such that $T_{a}(r)={\rm o} (T_{\bf g}(r))$.  By the basic properties of characteristic functions,  $K_{\bf g}$ is a field.  

\begin{theorem}\label{Mgcdunit}
Let $g_1,\hdots,g_n$ be  entire functions  without zeros and ${\bf g}=(1,g_1,\dots,g_n)$.  Let $F,G\in K_{\bf g}[x_1,\hdots,x_n]$ be nonconstant coprime polynomials.  Assume that  $g_1^{i_1} \cdots g_n^{i_n} \notin K_{\bf g}$ for any 
index set $(i_1,\dots,i_n)\in\mathbb Z^n\setminus \{(0,\dots,0)\}$. 
Let $\epsilon>0$.
\begin{enumerate}
\item  Then
$$
N_{\gcd}(F(g_1,\hdots,g_n), G(g_1,\hdots,g_n),r)\le_{\rm exc}  \epsilon\max_{1\le i\le n}\{ T_{g_i}(r)\}.
$$
\label{Mgcda}
\item If $F$ and $G$ are not both identically zero at the origin $(0,\dots,0)$ and the coefficients of $F$ and $G$ are entire functions in $K_{\bf g}$, then
$$
T_{\gcd}(F(g_1,\hdots,g_n), G(g_1,\hdots,g_n),r)\le_{\rm exc}  \epsilon\max_{1\le i\le n}\{ T_{g_i}(r)\}.
$$
\label{Mgcdb}
\end{enumerate}
\end{theorem}

\begin{theorem}\label{MgcdPn}
Let $g_1,\hdots,g_n$ be  meromorphic functions and ${\bf g}=[1:g_1:\dots:g_n]$ a holomorphic map from $\CC$ to  $\PP^n$.  Let $F,G\in K_{\bf g}[x_1,\hdots,x_n]$ be nonconstant coprime polynomials such that not both of them are identically zero at $(0,\dots,0)$.
If  $g_1^{i_1} \cdots g_n^{i_n} \notin K_{\bf g}$ for any 
index set $(i_1,\dots,i_n)\in\mathbb Z^n\setminus \{(0,\dots,0)\}$,
then for any $\epsilon>0$, there exists $k_0$ such that for $k\ge k_0$
\begin{enumerate}
\item   
$$
N_{\gcd}(F(g_1^k,\hdots,g_n^k), G(g_1^k,\hdots,g_n^k),r)\le_{\rm exc} \epsilon\max_{1\le i\le n}\{ T_{g_i^k}(r)\};
$$
\label{Mkgcda}
\item
$$
T_{\gcd}(F(g_1^k,\hdots,g_n^k), G(g_1^k,\hdots,g_n^k),r)\le_{\rm exc}  \epsilon\max_{1\le i\le n}\{ T_{g_i^k}(r)\},
$$
 if $g_1,\hdots,g_n$ are entire functions  and the coefficients of $F$ and $G$ are entire functions in $K_{\bf g}$.
\label{Mkgcdb}
\end{enumerate}
\end{theorem}

Let $g_1,\hdots,g_n$ be entire functions and ${\bf g}=(1,g_1,\dots,g_n)$.
Let $G\in K_{\bf g}[x_1, \dots,x_n ]$ be a nonconstant polynomial  such that $G(0,\dots,0)$ is not identically zero. Since $G$ has a non-zero constant term, after arranging the index set in some order, we may write  
$$
G=a_{{\bf i}(0)}+\sum_{j=1}^{\ell}a_{{\bf i}(j)}{\bf x}^{{\bf i}(j)}\in K_{\bf g}[x_1, \dots,x_n ],
$$      
where $a_{{\bf i}(j)}\ne 0$ for $0\le j\le \ell$.  We then let 
$\frak g_G:=(1,{\bf g}^{{\bf i}(1)},\hdots,{\bf g}^{{\bf i}(\ell)}):\CC\to \mathbb P^{\ell}.$ 
 \begin{corollary}\label{Mdivisibility}
 Let $g_1,\hdots,g_n$ be entire functions and ${\bf g}=(1,g_1,\dots,g_n)$. 
 Assume that $T_{\frak g_G}(r)\asymp  \max_{1\le i\le n}\{ T_{g_i}(r)\}$.
Let $F,\, G\in K_{\bf g}[x_1, \dots,x_n ]$ be nonconstant coprime  polynomials  with coefficients that are entire functions.  Assume that $G(0,\dots,0)$ is not identically zero.
 
\begin{enumerate}
\item If $F(g_1^k,\hdots,g_n^k)/G(g_1^k,\hdots,g_n^k)$ are entire functions for infinitely many positive integers $k$; or
\item $g_1,\hdots,g_n$ are entire functions without zeros and $F(g_1,\hdots,g_n)/G(g_1,\hdots,g_n)$ is an entire function,
\end{enumerate}
then
there exists an index set $(i_1,\dots,i_n)\in\mathbb Z^n\setminus \{(0,\dots,0)\}$ such that $g_1^{i_1} \cdots g_n^{i_n} \in K_{\bf g}$.
\end{corollary}

 \subsection{Nevanlinna Theory with Moving Targets}
Let ${\bf f}=(f_0,\dots,f_n)$ be a holomorphic map from $\CC$ to  $\PP^n$ where $f_0,f_1,\hdots,f_n$ are  holomorphic functions without a common zero.  
Let $a_{0}, \cdots,a_{n} \in K_{\bf f}$, and      
let $L:=a_{0} X_0+\dots+a_{n} X_n$.  
Then  $L$ defines a  hyperplane $H$ in $\PP^n(K_{\bf f})$.  We note that  $H(z)$ is the hyperplane determined by the linear form $L(z) =a_{0}(z) X_0+\dots+a_{n}(z) X_n$   for $z\in\CC$ that is not a common zero of $a_{0}, \cdots,a_{n}$, or a pole of any $a_{k}$, $0\le k\le n$.
The definition of the Weil function, proximity function and counting function can be easily extended to moving hyperplanes.  
For example,
\begin{align*}
    \lambda_{H(z)}(P)=-\log \frac{|(ha_{0})(z)x_0+\cdots+(ha_{n})(z)x_n|}{\max\{|x_0|,\dots ,|x_n|\}\max\{|(ha_{0})(z)|,\dots ,|(ha_{n})(z)|\}},
\end{align*}
where $h$ is a meromorphic function such that  $ha_{0}, \cdots,ha_{n} $ are entire functions without common zeros, $P=(x_0,\hdots,x_n)\in \PP^n(\CC)$ and $z\in\CC$.
It's clear that 
\begin{align}
    \lambda_{H(z)}(P)=-\log \frac{| a_{0} (z)x_0+\cdots+ a_{n} (z)x_n|}{\max\{|x_0|,\dots ,|x_n|\}\max\{| a_{0} (z)|,\dots ,| a_{n} (z)|\}},
\end{align}
for $z\in\CC$ which is not a a common zero of $a_{0}, \cdots,a_{n}$, or a pole of any $a_{k}$, $0\le k\le n$.
The first main theorem for a moving hyperplane $H$  can be stated as
\begin{equation}\label{fmtmov}
    T_{\mathbf{f}}(r) =N_{\mathbf{f}}(H,r)+m_{\mathbf{f}}(H,r)+{\rm o} (T_{\bf f}(r)).
\end{equation}

We will reformulate the second main theorem with moving targets stated in \cite[Theorem A4.2.1]{ru2001nevanlinna} to suit our purpose.
Let $a_{j0}, \ldots,a_{jn} \in K_{\bf f}$, and      
let $L_j:=a_{j0} X_0+\dots+a_{jn} X_n$.  
Without loss of generality, we will normalize the linear forms $L_j$, $1\le j\le q$, such that for each $1\le j\le q$, there exists $0\le j'\le n$ such that $a_{jj'}=1$.  Let $t$ be a positive integer and let $V(t)$ be the complex vector space spanned by the elements
\begin{align*}
\left\{ \prod a_{jk}^{n_{jk}} : n_{jk}\ge 0,\,\sum n_{jk}=t \right\},
\end{align*}
where the product and sum runs over $1\le j\le q$ and $0\le k\le n$.
Let $1=b_1,\cdots,b_u$ be a basis of $V(t)$ and $b_1,\cdots,b_w$ a basis of $V(t+1)$.  It's clear that $u\le w$.
Moreover, we have \cite[Lemma 6]{Wang}
\begin{align}
 \liminf_{t\to\infty}\dim V(t+1)/\dim V(t)=1. 
\end{align}
The following formulation of the second main theorem with moving targets follows from the proof of \cite[Theorem A4.2.1]{ru2001nevanlinna} by adding the Wronskian term when applying the second main theorem.
\begin{theorem}
\label{movingsmt}
    Let ${\bf f}=(f_0,\hdots,f_n):\mathbb{C}\to\mathbb{P}^n(\mathbb{C})$ be a holomorphic curve  where $f_0,f_1,\hdots,f_n$ are  entire functions without common zero.  Let $H_j$, $1\le j\le q$, be arbitrary (moving) hyperplanes given by $L_j:=a_{j0} X_0+\dots+a_{jn} X_n$ where $a_{j0}, \cdots,a_{jn}\in K_{\bf f}$.  Denote by $W$ the Wronskian of $\{hb_mf_k\,|\, 1\le m\le w,\,  0\le k\le n\}$, where $h$ is a meromorphic function such that $hb_1,\hdots, hb_w$ are entire functions without common zero.  If ${\bf f}$ is linearly non-degenerate over $K_{\bf f}$, then for any $\varepsilon>0$, we have the following inequality:
    $$ \int_0^{2\pi} \max_J \sum_{k\in J}\lambda_{H_k(re^{i\theta})}({\bf f}(re^{i\theta}))\frac{d\theta}{2\pi}+\frac 1uN_{W}(0,r)\leq_{\operatorname{exc}} \left(\frac wu(n+1)+\varepsilon\right)T_{\bf f}(r)+o(T_{\bf f}(r)),$$  where the maximum is taken over all subsets $J$ of $\{1,\dots, q\}$ such that $H_j(re^{i\theta})$, $j\in J$, are in general position.
\end{theorem}
 
The following two lemmas are moving targets versions of the Borel Lemma and Green's theorem. 
 
\begin{lemma}
    \label{Mborel1}
    Let $f_0,\hdots,f_n$ be entire functions with no zeros and  
$\mathbf{f}:=(f_0,\hdots,f_{n})$ 
be a holomorphic map from $\CC$ to $\PP^{n}(\CC)$.  
Suppose that $f_0,\hdots,f_n$  are linearly dependent over $K_{\bf f}$. Then for each $f_i$, there exists $j\ne i$ such that   $ f_i/f_j \in K_{\bf f}$.
\end{lemma}
The proof of this lemma is similar to the next one, therefore it is omitted.  
\begin{lemma}
    \label{Mgreen1}
    Let   $f_0,\hdots,f_n$ be non-zero entire functions without a common zero and let ${\bf f}=(f_0,\dots,f_n)$ be a holomorphic map from $\CC$ to  $\PP^n$.
Assume that   for an integer $k\ge n^2$ the following holds:
 \begin{align}\label{diageq}
  a_0f_0^k+\dots+a_nf_n^k=0,
  \end{align}
  where $a_i\ne 0\in K_{\bf f}$, $0\le i\le n$.
Then for each $f_i$, there exists $j\ne i$ such that   $(f_i/f_j)^k\in K_{\bf f}$.
\end{lemma}
The proof of this lemma can be found in \cite{Guo2}.  We include a slightly different proof for completeness.
 \begin{proof}
By reindexing the $f_i$, we may assume that $f_0^k,  f_1 ^k,\cdots, f_{m-1}^k$, for some $1\le m\le n$, is a basis of the $K_{\bf f}$-vector space  spanned by $f_0^k,  f_1 ^k,\cdots, f_{n}^k$.     It suffices to show that for $m\le i\le n$, there exists $0\le j\le m-1$ such that $(f_i/f_j)^k\in K_{\bf f}$ as $f_0^k,  f_1 ^k,\cdots, f_{m-1}^k$ are $K_{\bf f}$-linearly independent and the $f_i$ satisfy \eqref{diageq}.
Let $m\le i\le n$.  After possibly reindexing, we obtain
$$
 f_{i}^k=\alpha_0f_0^k+\alpha_{1}f_1 ^k+\dots +\alpha_{\ell} f_{\ell}^k,
$$
where $0\le \ell\le m-1$ and $\alpha_0,\ldots,\alpha_{\ell}\in K_{\bf f}\setminus\{0\}$.   If $\ell=0$, then   $(f_i/f_j)^k\in K_{\bf f}$ for some $0\le j\le m-1$.  Suppose that $\ell\ge 1$. 
Let $\beta$ be an entire function such that $\tilde f_0=f_0/\beta,\hdots,  \tilde f_{\ell}=f_{\ell}/\beta$ have no common zero, let $\tilde f_i=f_i/\beta$, and let
$$
\tilde{\mathbf f}_k:=[f_0^k:f_1^k:\cdots:f_{\ell}^k]= (\tilde f_0^k,\tilde f_1^k,\hdots,\tilde f_{\ell}^k).
$$
Let $h$ be a meromorphic function such that $h\alpha_0\tilde f_0^k  ,\hdots,h\alpha_{\ell}\tilde f_{\ell}^k $ are entire functions with no common zeros, and let 
$$
\mathbf{F}_k:=(h\alpha_0\tilde f_0^k  ,\cdots,h\alpha_{\ell}\tilde f_{\ell}^k)
$$ be a holomorphic map from $\CC$ to $\PP^{\ell}(\CC)$.
Then
\begin{align*}
T_{\mathbf{F}_k}(r) &\le  T_{\tilde{\mathbf f}_k}(r) +\sum_{j=0}^{\ell}m_{\alpha_j}(\infty, r)+N_h(0,r)\cr
&\le T_{\tilde{\mathbf f}_k}(r)+\sum_{j=0}^{\ell}T_{\alpha_j}(r)+\sum_{j=0}^{\ell}N_{\alpha_j}(\infty,r)\cr
&\le  T_{\tilde{\mathbf f}_k}(r)+{\rm o}( T_{{\mathbf f}}(r)). 
\end{align*}
Similarly, by writing  $\tilde{\mathbf f}_k=((h\alpha_0)^{-1}h\alpha_0\tilde f_0^k,\hdots,(h\alpha_{\ell})^{-1}h\alpha_{\ell}\tilde f_{\ell}^k)$, we have
\begin{align}\label{charF}
T_{\tilde{\mathbf f}_k}(r)\le T_{\mathbf{F}_k}(r)+{\rm o}( T_{{\mathbf f}}(r)). 
\end{align}
 Moreover, the map $\mathbf{F}_k$ is linearly non-degenerate over $\CC$,
since  $f_0^k,  f_1 ^k,\cdots, f_{\ell}^k$ are $K_{\bf f}$-linearly independent.  Applying Theorem \ref{tsmt} to the map $\mathbf{F}_k$ with the coordinate hyperplanes $\{X_j=0\}$, $0\le j\le \ell$, and the diagonal hyperplane $\{X_0+\cdots+X_{\ell}=0\}$ of  $\PP^{\ell}(\CC)$,  for any $\epsilon>0$ we have
\begin{align*}
 T_{\mathbf{F}_k}(r)&\le_{\rm exc}  \sum_{j=0}^{\ell} N_{h\alpha_j \tilde f_j^k }^{(\ell)}(0,r)+N_{h \tilde f_i^k }^{(\ell)}(0,r) + \epsilon T_{\mathbf{F}_k}(r)\cr
       &\le_{\rm exc} \sum_{j=0}^{\ell} N_{\alpha_j}(0,r)+\ell\sum_{j=0}^{\ell  } N_{\tilde f_j}(0,r)+\ell N_{\tilde f_i}(0,r)+(\ell+2)N_h(0,r) +\epsilon T_{\mathbf{F}_k}(r)\cr
      &\le_{\rm exc} \frac{\ell}{k}(\ell+2) T_{\tilde{\mathbf f}_k}(r) +\epsilon T_{\mathbf{F}_k}(r)+{\rm o}( T_{{\mathbf f}}(r)).
\end{align*} 
Together with \eqref{charF}, for all $\epsilon>0$ we have 
   \begin{equation*} 
        \begin{aligned}
            kT_{\tilde{\mathbf f}_k}(r)\le_{\rm exc} (\ell^2+2\ell+\epsilon) T_{\tilde{\mathbf f}_k}(r)+{\rm o}( T_{{\mathbf f}}(r)).
            \end{aligned}
    \end{equation*} 
       Hence,
         \begin{equation*}
   (k-\ell^2-2\ell-\epsilon) T_{\tilde{\mathbf f}_k}(r)\le_{\rm exc} {\rm o}( T_{{\mathbf f}}(r)).
    \end{equation*}
 If $k\ge  n^2\ge(\ell+1)^2>\ell^2+2\ell$, then this implies that  $T_{ \frac{f_j}{f_0} }(r)\le T_{\tilde{\mathbf f}_k}(r)\le_{\rm exc} {\rm o}( T_{{\mathbf f}}(r))$ for $1\le j\le \ell$, and hence  $ \frac{f_j}{f_0}\in K_{\bf f}$ for $1\le j\le \ell$, contradicting that $f_0^k,  f_1 ^k,\cdots, f_{\ell}^k$ are linearly independent over $K_{\bf f}$.
   \end{proof}

\subsection{Key Theorem}
The following fundamental result is the analogue of Theorem \ref{Refinement}.

\begin{theorem}\label{Mfundamental}
Let $g_0,g_1,\dots,g_n $ be entire functions without common zeros and let ${\bf g}=(g_0,g_1,\dots,g_n)$.  
Let $F,G\in K_{\bf g}[x_0,x_1,\hdots,x_n]$ be coprime homogeneous polynomials of the same degree $d>0$.  Let $I$ be the set of exponents ${\bf i}$ such that ${\bf x}^{\bf i}$ appears with a nonzero coefficient in either $F$ or $G$.  Let $m\ge d$ be a positive integer.   Suppose that the set $\{g_0^{i_0}\dots g_n^{i_n}: i_0+\cdots+i_n=m\}$ is linearly independent over $K_{\bf g}$.  Then for any $\epsilon>0$, there exists a positive integer $L$ such that the following holds:
\begin{align*}
& MN_{\rm gcd}(F({\bf g}),G({\bf g}),r) \\
&\le_{\rm exc} c_{m,n,d}  \sum_{i=1}^n N^{(L)}_{ g_i}(0,r)+ \left(\frac{m}{n+1}\binom{m+n}{n}-c_{m,n,d}-M'm\right)\sum_{i=1}^nN_{g_i}(0,r)\\ 
  &+\binom {m+n-2d}{n}N_{\rm gcd}(\{{\bf g}^{\bf i}\}_{{\bf i}\in I},r)+ \left(M'mn+\epsilon m +\frac {M\epsilon}2\right)T_{\bf g}(r)+{\rm o}(T_{\bf g}(r)),
\end{align*}
where $c_{m,n,d}=2\binom {m+n-d}{n+1}- \binom {m+n-2d}{n+1}$,
$M=2\binom {m+n-d}{n}- \binom {m+n-2d}{n}$, and $M'$ is an integer of order $O(m^{n-2})$.
\end{theorem}
The basic ideas used to prove the theorem are similar to the ideas used in the proof of Theorem~\ref{Refinement}.  
We will make explicit the important differences in the moving target case and omit whatever is
  identical or obvious from the proof of  Theorem \ref{Refinement}.
\begin{proof}[Proof of Theorem \ref{Mfundamental}]
Let $(F,G)$ be the ideal generated by $F$ and $G$ in $K_{\bf g}[x_0, \cdots,x_n ]$ and  $(F,G)_m:=K_{\bf g}[x_0, \cdots,x_n ]_m\cap (F,G).$
We choose $\{\phi_1,\hdots,\phi_M\}$ to be a basis of the $K_{\bf g}$-vector space $(F,G)_m$ consisting of elements of the form $F{\bf x}^{\bf i}$ and $G{\bf x}^{\bf j}$.  

For each $z\in \CC$, we construct a basis $B_z$ for $V_m=K_{\bf g}[x_0,\hdots,x_n]_m/(F,G)_m$ as follows.  
For ${\bf i}=(i_0,\cdots,i_n)\in\mathbb{N}^{n+1}$, we let ${\bf g}(z) ^{\bf i}   =g_0(z)^{i_0}\cdots g_n(z)^{i_n}$.
Choose a monomial ${\bf x}^{{\bf i}_1}$   so that    $ |{\bf g}(z)^{{\bf i}_1}| $ is minimal subject to the condition ${\bf x}^{{\bf i}_1}\notin (F,G).$  Suppose now that ${\bf x}^{{\bf i}_1},\hdots,{\bf x}^{{\bf i}_j}$ have been constructed.  Then we let ${\bf x}^{{\bf i}_{j+1}}\in K_{\bf g}[x_0,\hdots,x_n]_m$ be a monomial such that $|{\bf g}(z) ^{{\bf i}_{j+1}}|$ is minimal subject to the condition that 
${\bf x}^{{\bf i}_1},\hdots,{\bf x}^{{\bf i}_{j+1}}$ are linearly independent modulo $(F,G)_m$.  In this way, we construct a basis of $V_m$ with monomial representatives ${\bf x}^{{\bf i}_1},\hdots,{\bf x}^{{\bf i}_{M'}}$, where $M'=\dim V_m$.
Let $I_z=\{ {\bf i}_1,\hdots,{\bf i}_{M'}\}$.  Then for each ${\bf i}$, $|{\bf i}|= m$, we have
\begin{align*}
{\bf x}^{\bf i}+\sum_{j=1}^{M'}b_{{\bf i},j,z}{\bf x}^{{\bf i}_j}\in (F,G)_m
\end{align*}
for some choice of coefficient $b_{{\bf i},j,z}\in K_{\bf g}$.  
Then for each such ${\bf i}$ there is a linear form $L_{z,{\bf i}}$ over $K_{\bf g}$ such that 
\begin{align}\label{coefficient}
L_{z,{\bf i}}(\phi_1,\hdots,\phi_M)= {\bf x}^{\bf i}+\sum_{j=1}^{M'}b_{{\bf i},j,z}{\bf x}^{{\bf i}_j}.
\end{align}
We note that there are only finitely many choice of $b_{{\bf i},j,z}$ even as $z$ runs through all of $\CC$, and 
 $\{L_{z,{\bf i}}(\phi_1,\hdots,\phi_M) \,|\, |{\bf i}|= m, {\bf i}\notin I_z\}$ is a basis for $(F,G)_m$.
From the definition of ${\bf x}^{{\bf i}_1},\hdots,{\bf x}^{{\bf i}_{M'}}$, we have the key inequality
\begin{align}\label{keyinequality2m}
\log |L_{z,{\bf i}}(\Phi ({\bf g}(z))|\le \log |{\bf g}(z)^{{\bf i} }|+ \log\| L_{z,{\bf i}}(z)\|+O(1),
\end{align}
where $\Phi ({\bf g}(z)) =(\phi_1({\bf g}(z)),\hdots,\phi_M({\bf g}(z)))$, $\| L_{z,{\bf i}}(z)\|=\max_{0\le j\le M'} |b_{{\bf i},j,z}(z)|$, and $b_{{\bf i},0,z}=1$.  
The  map $ (\phi_1({\bf g} ),\hdots,\phi_M({\bf g} ) )$ may not be a reduced presentation of $\Phi ({\bf g} )$.  Let $h$ be a meromorphic function such that $F({\bf g} )/h$ and $G({\bf g})/h$ are entire and have no common zeros, i.e., 
\begin{align}\label{hzero}
N_h(0,r)=N_{\rm gcd}(F({\bf g}),G({\bf g}),r).
\end{align}     Moreover, since $g_i$, $0\le i\le n$, are entire functions, the poles of $h$ comes from the poles of the coefficients of $F$ and $G$, and hence
\begin{align}\label{hpole}
N_h(\infty ,r)\le {\rm o}(T_{\bf g}(r)).
\end{align}    
Let $\psi_i:=\phi_i({\bf g})/h $, $1\le i\le M$.  By the choice of  $\phi_i $,  the function $\psi_i$ is entire for $1\le i\le M$.  Furthermore, since $FX_i^{m-d}, GX_i^{m-d}\in (F,G)_m$, $0\leq i\leq n$, and $g_0,\ldots, g_n$ have no common zero, the functions $\psi_i$, $1\le i\le M$, have no common zero.  Hence  $\Psi:=(\psi_1 ,\hdots,\psi_M )$ is a reduced form of $\Phi ({\bf g})$.
 
Let $V$ be the complex vector space spanned by $1$, the coefficients of $F$ and $G$, and all possible (finitely many) choices of $b_{{\bf i},j,z}$ in (\ref{coefficient}).
Let $t$ be a large positive integer and let $V(t)$ be the finite-dimensional vector space spanned by products of $t$ elements in $V$.

Let $b_1=1,b_2,\cdots,b_u$ be a basis of $V(t)$ and $b_1,\cdots,b_w$ be a basis of $V(t+1)$. 
We will choose $t$ sufficiently large so that 
$\frac wu\le 1+\frac{\epsilon}{2m}$.  Denote by $W$ the Wronskian of $\{a b_m\psi_k\,|\, 1\le m\le w,\,  1\le k\le M\}$, where $a$ is an entire function such that $a=a b_1, a b_2,\cdots,ab_w$  are entire and have no common zeros.  Applying Theorem \ref{movingsmt} to the holomorphic map $\Phi ({\bf g})$ with the reduced form $\Psi=(\psi_1 ,\hdots,\psi_M )$ and the set of linear forms $\{L_{z,{\bf i}}  \,|\, |{\bf i}|= m, {\bf i}\notin I_z\}$, we have the following inequality: 
\begin{align}\label{useMSMT}
 \int_0^{2\pi}  \sum_{|{\bf i}|= m, {\bf i}\notin I_z}\lambda_{L_{z,{\bf i}}(\Psi(re^{i\theta}))}(\Psi ( re^{i\theta}))
   \frac{d\theta}{2\pi} +\frac 1u N_{W}(0,r)\leq_{\operatorname{exc}} (\frac wu M+\varepsilon)T_{\Psi}(r)+{\rm o}(T_{\Psi }(r)),
\end{align} 
where
\begin{align}\label{lambda}
\lambda_{L_{z,{\bf i}}(\Psi(re^{i\theta}))}(\Psi ( re^{i\theta}))=-\log |L_{z,{\bf i}}(\Phi ({\bf g}(z))|+\log\|\Phi ({\bf g}(z))\| +\log\|L_{z,{\bf i}}(z)\|,
\end{align}
if $z=re^{i\theta}$ is not a pole of any coefficient of $L_{z,{\bf i}}$.  We note that such $z$ is in a discrete subset of $\CC$, and hence its radius $r\in (0,\infty)$ is in a  set of finite Lebesgue measure.
In the following computation, we will only consider $z$ which is not a pole of any $b_{{\bf i},j,z}$ from \eqref{coefficient}.
Next, we will derive a lower bound for the first term of the left hand side of (\ref{useMSMT}).  By the definition of the characteristic function, Lemma \ref{Jensen}, \eqref{hzero} and \eqref{hpole}, we have
\begin{align*} 
\int_0^{2\pi}   \log\|\Phi ({\bf g}(z)) )\|  \frac{d\theta}{2\pi} 
&= \int_0^{2\pi}\log \max\{|\psi_1(z),\hdots, \psi_M(z)|\}+\log| h(z)| \frac{d\theta}{2\pi}\cr
&= T_{\Psi}(r)+ N_{\rm gcd}(F({\bf g}),G({\bf g}),r)+{\rm o}(T_{\bf g}(r)).
\end{align*}

On the other hand, since $\phi_i=\sum_{I}\alpha_Ix^I\in K_{\bf g}[x_0,\hdots,x_n]_m$,
$$
\log|\phi_i({\bf g}(z))|\le m\log\max\{|g_0(z)|,\hdots, |g_n(z)|\}+\sum_{I}\log^+|\alpha_I(z)|+O(1),
$$
and hence
\begin{align*}
 \int_0^{2\pi}  \log\|\Phi ({\bf g}(z))\|  \frac{d\theta}{2\pi}
 \le m  T_{\bf g}(r)+{\rm o}( T_{\bf g}(r)).
\end{align*}
In conclusion, we have
\begin{align}\label{Mheight}
N_{\rm gcd}(F({\bf g}),G({\bf g}),r)+T_{\Psi}(r)=\int_0^{2\pi}   \log\|\Phi ({\bf g}(z))\|  \frac{d\theta}{2\pi} +{\rm o}( T_{\bf g}(r)) \le m T_{\bf g}(r)+{\rm o}( T_{\bf g}(r)).
\end{align} 

Now, we use the key inequality (\ref{keyinequality2m}),   the estimates in \eqref{estimateL} and the two equations following it to derive
\begin{align}\label{MestimateL}
   \sum_{|{\bf i}|= m, {\bf i}\notin I_z} &(-\log |L_{z,{\bf i}}(\Phi ({\bf g}(z)))| +\| L_{z,{\bf i}}(z)\|)
 \ge  -\sum_{|{\bf i}|= m, {\bf i}\notin I_z} \log | {\bf g}(z)^{{\bf i} }|+O(1)\cr
&\ge-\frac{m}{n+1}\binom{m+n}{n}\sum_{i=0}^n\log|g_i(z)|+M'm\sum_{i=0}^n\log |g_i(z)|\cr
&\qquad\qquad-M'mn\max\{\log|g_0(z)|,\hdots,\log|g_n(z)|\}+O(1),
\end{align}
for  $z\in\CC$  that  is not a pole of any $b_{{\bf i},j,z}$ from \eqref{coefficient}.
By Lemma \ref{Jensen}, Jensen's formula, the integration of (\ref{MestimateL}) from 0 to $2\pi$ over 
$d\theta$ gives
 \begin{align*}
\int_0^{2\pi}  &\sum_{|{\bf i}|= m, {\bf i}\notin I_z} (-\log |L_{z,{\bf i}}(\Phi ({\bf g}(z) ))|+\| L_{z,{\bf i}}(z)\| ) \frac{d\theta}{2\pi}\cr
&\ge_{\operatorname{exc}}  -\left(\frac{m}{n+1}\binom{m+n}{n}-M'm\right)\sum_{i=0}^nN_{g_i}(0,r)-M'mnT_{\bf g}(r)+{\rm o}( T_{\bf g}(r)).
\end{align*}

Together with (\ref{useMSMT}) and  (\ref{Mheight}), we have
\begin{align}\label{lambdaI}
MN_{\rm gcd}(F({\bf g}),G({\bf g}),r)&+\frac 1{u} N_{W}(0,r)\leq_{\operatorname{exc}} \left(\left(\frac wu-1\right)M+M'n +\epsilon \right) mT_{\bf g}(r) \cr
 & +\left(\frac{m}{n+1}\binom{m+n}{n}-M'm\right)\sum_{i=0}^nN_{g_i}(0,r) +{\rm o}( T_{\bf g}(r)).
\end{align}
Since $F$ and $G$ are coprime, the ideal $(F,G)$ defines a closed subset of $\mathbb A^n$ of codimension at least 2.  As is well-known from the theory of Hilbert functions and Hilbert polynomials, this implies that $M'=O(m^{n-2})$.

It then suffices to show that there exists a large integer $L$ (to be determined later) such that 
  \begin{equation}
    \label{Mtruncation2}
        \begin{aligned}
        c_{m,n,d}\sum_{i=0}^n N_{g_i}(0,r) - \binom {m+n-2d}{n}N_{\rm gcd}(\{{\bf g}^{\bf i}\}_{{\bf i}\in I},r)-\frac 1uN_{W }(0,r)\le   c_{m,n,d}\sum_{i=0}^n N^{(L)}_{g_i}(0,r).
        \end{aligned}
    \end{equation} 
The above inequality can be deduced from the inequality
  \begin{equation}
    \label{Mlocaleq2}
        \begin{aligned}
       c_{m,n,d}\sum_{i=0}^n  v_{z}^+ ( g_i)- \binom {m+n-2d}{n}\min_{{\bf i}\in I}v_z^+({\bf g}^{\bf i}) -\frac1uv_{z}^+ ( W )\le  c_{m,n,d}\sum_{i=1}^n \min\{L,v_{z}^+  (g_i)\},
        \end{aligned}
    \end{equation} 
 for $z\in\CC$.  The inequality holds trivially if $v_{z}^+(g_i)\le  L$ for each $1\le i\le n$.
Therefore, we only need to consider the case where $v_{z}^+(g_i)> L$ for some $1\le i\le n$.

For $z\in\CC$, we define a monomial ordering $>_{{\bf g}(z)}$ on $A=K_{\bf g}[x_0,\cdots,x_n]$ using the weight vector ${\bf u}=(v_z(g_0),\ldots, v_z(g_n))$, where we set $g_0=1$.  Let
  \begin{align*}
B_1&=\{F_1{{\bf x}^{\bf i}}:\, |{\bf i}|=m-d\},\\
B_2&=\{F_2{{\bf x}^{\bf i}}:\, |{\bf i}|=m-d\},\\
B'_1&=\{F_1{\rm TM}_{{\bf g}(z)}(F_2){\bf x}^{\bf i}:\, |{\bf i}|=m-2d\},
\end{align*}
where $\{F_1,F_2\}=\{F,G\}$ and ${\rm TM}_{{\bf g}(z)}(F_2)\le {\rm TM}_{{\bf g}(z)}(F_1)$.
By Lemma \ref{mainlemma},
$B=(B_1\setminus B_1')\cup B_2$
is a basis for $(F,G)_m$. 
Write $B=\{\beta_1,\hdots,\beta_M\}$, which depends on $z$, although there are only a finite number of choices of such a basis.
Let $\eta_j= \beta_j({\bf g})/h$ (note that the coefficients of the $\beta_j$  come from the coefficients of $F$ and $G$).  From the definition of $>_{{\bf g}(z)}$ and $F_2$, it follows that 
\begin{align*}
v_z^+({\rm TM}_{{\bf g}(z)}(F_2))=\min_{{\bf i}\in I}v_z^+({\bf g}^{\bf i}), 
\end{align*}
where $I$ is the set of exponents ${\bf i}$ such that ${\bf x}^{\bf i}$ appears with a nonzero coefficient in either $F$ or $G$.
Then similar to \eqref{BS},  by the second part of Lemma \ref{mainlemma}, we have  for each $z\in \CC$,
\begin{align}\label{MBS} 
\sum_{j=1}^M  v_{z}^+ (\eta_j)  
\geq  c_{m,n,d}\sum_{i=1}^{n}v_{z}^+ (g_i)-\binom {m+n-2d}{n}\min_{{\bf i}\in I}v_z^+({\bf g}^{\bf i}).
\end{align} 

On the other hand, from the definition of the $\phi_i$ and $\beta_i$, we see that 
$\eta_j$ is a  $V$-linear combination of 
$\psi_k$, $1\le k\le M$.  Therefore, $ab_m\eta_j$, $1\le m\le u$, $1\le j \le M$, is a $\CC$-linear combination of 
$\{a b_m\psi_k\,|\, 1\le m\le w,\,  1\le k\le M\}$.
As $W$ is the Wronskian of $\{a b_m\psi_k\,|\, 1\le m\le w,\,  1\le k\le M\}$, where $a$ is an entire function such that $a=a b_1, a b_2,\cdots,ab_w$  are entire and have no common zeros, from the basic properties of Wronskians we have
 \begin{equation}
    \label{MordW2}
        \begin{aligned}
     v_{z}^+(W)&\ge  \sum_{m=1}^u\sum_{j=1}^M  v_{z}^+ (ab_m\eta_j)  -\frac12Mw(Mw-1)\cr
     &\ge  u\sum_{j=1}^M  v_{z}^+ (\eta_j)  -\frac12Mw(Mw-1)\quad\text{(since $ab_m$ is entire for each $m$)}.        
      \end{aligned}
 \end{equation}  
 
Combining \eqref{MBS}  and \eqref{MordW2}, we obtain that
\begin{equation}
    \label{Mlocal2}
        \begin{aligned}
         & c_{m,n,d}\sum_{i=0}^{n}v_{z}^+ (g_i)-\binom {m+n-2d}{n}\min_{{\bf i}\in I}v_z^+({\bf g}^{\bf i})- \frac1uv_{z}^+(W(\Psi))\cr
 &\le  \frac 1{2u}  Mw (Mw-1) .        
        \end{aligned}
    \end{equation}
Let $L=\frac 1{2u}  Mw (Mw-1)c_{m,n,d}^{-1} $.   
The assumption that $v_{z}^+(g_i)> L$ for some $0\le i\le n$ implies that 
\begin{align}\label{ML2}
\frac 1{2u}  Mw(Mw-1)=c_{m,n,d}L\le c_{m,n,d}\sum_{i=0}^n \min\{L,v_{z}^+ (g_i)\}.
\end{align}

Therefore, the inequality (\ref{Mlocaleq2}) can be deduced from (\ref{Mlocal2}) and (\ref{ML2}).
 \end{proof}

\subsection{Proof of the main theorems}
The following theorem is the moving target version of Theorem \ref{Theorem10.3}.
Denote by   $R_{\bf g}\subset K_{\bf g}$,  the subring of entire functions
\begin{theorem}\label{MTheorem10.3}
Let $g_1,\hdots,g_n$ be  entire functions  without common zeros and ${\bf g}=(1,g_1,\dots,g_n)$. Let  $G\in R_{\bf g}[x_1,\hdots,x_n]$ be a nonconstant polynomial of degree $d$ such that $G(0,\hdots,0)$ is not identically zero.   Assume that  $g_1^{i_1} \cdots g_n^{i_n} \notin K_{\bf g}$  for any 
index set $(i_1,\dots,i_n)\in\mathbb Z^n\setminus \{(0,\dots,0)\}$.  For all $\epsilon>0$,  
\begin{enumerate}
\item there exists a positive integer $k_0$ such that for all $k\ge k_0$,
$$
m_{G(g_1^k,\hdots,g_n^k)}(0,r)\le_{\rm exc}\epsilon  \sum_{i=1}^n T_{g_i^k}(r);
$$
\item if in addition each $g_i$ has no zero, $1\le i\le n$, then 
$$
m_{G(g_1,\hdots,g_n)}(0,r)\le_{\rm exc}\epsilon  \sum_{i=1}^n T_{g_i}(r).
$$
\end{enumerate}
\end{theorem}
 \begin{proof} 
By  Lemma \ref{Mgreen1}, the set $\{{\bf g}^{k{\bf i}}:=g_1^{ki_1}\cdots g_n^{ki_n}\,|\, {\bf i}=(i_1,\hdots,i_n)\in \mathbb{N}^n, |{\bf i}|\le d\}$ 
is $K_{\bf g}$-linearly independent for $k\ge \binom {d+n}{n}^2$.  Since $G(0,\hdots,0)\ne 0$, by  arranging the index set in some order, we may write  
$$
G=a_{{\bf i}(0)}+\sum_{j=1}^{\ell}a_{{\bf i}(j)}{\bf x}^{{\bf i}(j)},
$$  
where $a_{{\bf i}(j)}\in R_{\bf g}\setminus \{0\}$ for $0\le j\le \ell$.
Then we have  
$$
G(g_1^k,\hdots,g_n^k)=a_{{\bf i}(0)}+\sum_{j=1}^{\ell}a_{{\bf i}(j)}{\bf g}^{k{\bf i}(j)}.
$$  
Let $h$ be an entire function such that $h^{-1} a_{{\bf i}(0)}, h^{-1} a_{{\bf i}(1)}{\bf g}^{k{\bf i}(1)},\hdots,h^{-1} a_{{\bf i}(\ell)}{\bf g}^{k{\bf i}(\ell)}$ have no common zero.  Hence,
\begin{align*}
N_h(0,r)\le N_{a_{{\bf i}(0)}}(0,r)\le T_{a_{{\bf i}(0)}}(r)\le {\rm o}( T_ {\bf g}(r)).
\end{align*}
We apply  Theorem \ref{tsmt}, Cartan's truncated second main theorem, to the holomorphic map 
$$
\frak g(k):=(h^{-1}a_{{\bf i}(0)}, h^{-1}a_{{\bf i}(1)}{\bf g}^{k{\bf i}(1)},\hdots,h^{-1}a_{{\bf i}(\ell)}{\bf g}^{k{\bf i}(\ell)}):\CC\to \mathbb P^{\ell}
$$ associated with the above expression ($k\ge \binom {d+n}{n}^2$), and with the set of coordinate hyperplanes of $\mathbb P^{\ell}$ and the diagonal hyperplane $\sum_{i=0}^{\ell}X_j$. Then for any $\epsilon>0$, we have
\begin{align*} 
(1-\frac \epsilon2)T_{\frak g(k)}\le_{\operatorname{exc}} N^{(\ell)}_{G(g_1^k,\hdots,g_n^k)}(0,r)+\sum_{j=1}^{\ell} N^{(\ell)}_{{\bf g}^{k{\bf i}(j)}}(0,r)+{\rm o}( T_ {\bf g}(r)).
\end{align*}
We note that
$$
N^{\ell}_{{\bf g}^{k{\bf i}(j)}}(0,r)\le  \frac{\ell}k N_{{\bf g}^{k{\bf i}(j)}}(0,r)\le  \frac{\ell}k T_{{\bf g}^{k{\bf i}(j)}}(r)\le \frac{d\ell}k\max_{1\le i\le n}\{T_{g_i^k}(r)\}+{\rm o}( T_ {\bf g}(r)).
$$
 Then for $k>\max\{\frac{2\ell^2}{\epsilon}, \binom {d+n}{n}^2\}$,
\begin{align}\label{Musetsmta}
(1- \epsilon )T_{\frak g(k)}(r)&\le_{\operatorname{exc}} N_{G(g_1^k,\hdots,g_n^k)}(0,r)+{\rm o}( T_ {\bf g}(r))\cr
&=T_{G(g_1^k,\hdots,g_n^k)}(r)-m_{G(g_1^k,\hdots,g_n^k)}(0,r)+{\rm o}( T_ {\bf g}(r)).
\end{align}
Since  $G(g_1^k,\hdots,g_n^k)$ is an entire function, 
\begin{align}
T_{G(g_1^k,\hdots,g_n^k)}(r)&=m_{G(g_1^k,\hdots,g_n^k)}(\infty,r)\cr
&\le  T_{\frak g(k)}+N_{h}(0,r)+O(1)\le  T_{\frak g(k)}+{\rm o}( T_ {\bf g}(r)).
\end{align}
Consequently,
$$
m_{G(g_1^k,\hdots,g_n^k)}(0,r)\le_{\operatorname{exc}}  \epsilon T_{\frak g(k)}(r)\le_{\operatorname{exc}} d \epsilon T_ {(1,g_1^k,\cdots,g_n^k)}(r).
$$

When  the $g_i$ are entire functions without zeros, we may assume  that the set $\{{\bf g}^{ {\bf i}}:=g_1^{ i_1}\cdots g_n^{ i_n}\,|\, {\bf i}=(i_1,\hdots,i_n)\in \mathbb{N}^n, |{\bf i}|\le d\}$ 
is linearly independent over $K_{\bf g}$ by Lemma \ref{Mborel1}.  Then we can repeat the previous argument  for $k=1$ with the additional condition $N_{g_i}(0,r)=O(1)$, $1\le i\le n$, to conclude the proof.
 \end{proof}

The proofs of Theorem   \ref{Mgcdunit}, Theorem \ref{MgcdPn} and Corollary \ref{Mdivisibility} are similar to their constant analogues as the necessary tools have been developed.  Therefore, we will not repeat the proofs here.


 \end{document}